\documentclass[10pt]{amsart}

\usepackage{hyperref}

\newcommand{\R}{\mathbb{R}}
\newcommand{\C}{\mathbb{C}}
\newcommand{\BB}{\mathcal{B}}

\newcommand{\FF}{\mathcal{F}}
\newcommand{\dd}{\;{\rm d}}
\newcommand{\de}{{\rm d}}

\newcommand{\HH}{\mathcal{H}}

\newcommand{\LL}{\mathcal{L}}

\newcommand{\N}{\mathbb{N}}
\newcommand{\integer}{\mathbb{Z}}

\DeclareMathOperator{\Leb}{Leb}
\DeclareMathOperator{\dLeb}{dLeb}

\newcommand{\Z}{\mathbb{Z}}
\newcommand{\T}{\mathcal{T}}
\newcommand{\st}{\;|\;}
\newcommand{\ii}{\mathbf{i}}

\newcommand{\Lp}{\mathcal{L}}

\newcommand{\real}{\mathbb{R}}
\newcommand{\complex}{\mathbb{C}}
\newcommand{\LSRB}{\Lp_{1/|\det DT|}}
\newcommand{\Cs}{C_{\#}}
\newcommand{\Pp}{\mathcal{P}}

\newcommand{\norm}[2]{\left\| #1 \right\|_{#2}}
\newcommand{\nor}[1]{\left\| #1 \right\|}
\DeclareMathOperator{\Card}{Card}
\DeclareMathOperator{\supp}{supp} \DeclareMathOperator{\Id}{Id}
\newcommand{\transposee}[1]{{\vphantom{#1}}^{\mathit t}{#1}}

\newtheorem{thm}{Theorem}
\newtheorem{prop}[thm]{Proposition}
\newtheorem{definition}[thm]{Definition}
\newtheorem{lem}[thm]{Lemma}
\newtheorem{cor}[thm]{Corollary}
\newtheorem*{prop*}{Proposition}

\theoremstyle{definition}
\newtheorem{ex}{Example}
\newtheorem{rmk}[thm]{Remark}

\begin{document}

\title[Spaces for piecewise hyperbolic maps]{Good Banach
spaces for piecewise hyperbolic maps via interpolation}
\author{Viviane Baladi and S\'{e}bastien Gou\"{e}zel}
\address{D.M.A., UMR 8553, \'{E}cole Normale Sup\'{e}rieure,  75005 Paris, France}
\email{viviane.baladi@ens.fr}

\address{IRMAR,
Universit\'{e} de Rennes 1, 35042 Rennes, France}
\email{sebastien.gouezel@univ-rennes1.fr}
\thanks{We are very grateful to W. Sickel and A. Baghdasaryan for helpful
comments about the literature.  VB did part of this work
at UMI 2924 CNRS-IMPA, Rio de Janeiro.
VB is partially supported by ANR-05-JCJC-0107-01.}

\begin{abstract}
We introduce a weak transversality condition
for  piecewise $C^{1+\alpha}$
and piecewise hyperbolic maps which admit a $C^{1+\alpha}$ stable
distribution.
We show good bounds on the essential spectral radius of the associated
transfer operators acting on classical anisotropic
Sobolev spaces of Triebel-Lizorkin type. In many cases, we obtain
a spectral gap from which we deduce the existence of finitely many
physical measures with basin of total measure.
The analysis relies on standard techniques
(in particular complex interpolation) and
applies also to piecewise expanding maps
and to Anosov diffeomorphisms, giving a unifying picture
of several previous results.
\end{abstract}
\date{November 12, 2007}
\maketitle

Proving the existence of physical measures and studying their
statistical properties is an important task in dynamical
systems. In this paper, we shall be concerned with
maps with singularities (that is, discontinuities in the
map or its derivatives). We shall assume that the map is piecewise
smooth relative to a finite partition, and the
most challenging case is when this partition does not have a
Markov-type property.

For one-dimensional piecewise expanding maps, the space of
functions of bounded variation has proved a very powerful tool,
since the transfer operator acting on it has a spectral gap.
This readily implies the existence of finitely many  physical
measures whose basins have full measure, as well as numerous
other consequences.
 This functional approach
has been extended to higher dimensional piecewise expanding
maps, under  stronger assumptions (the counter-examples of
Tsujii \cite{Ts00} and Buzzi \cite{Bu01} show that {\it some}
additional assumption is necessary), by considering various
functional spaces (see the work of Keller, G\'{o}ra--Boyarski,
Saussol, Buzzi, Tsujii, Cowieson  \cite{Ke, GB, Sau, Bu00,
Ts00', Co}). On the other hand, a more elementary approach,
involving a more detailed study of the dynamics and how sets
are cut by the discontinuities, was developed by Young and
Chernov \cite{Yo, Ch}, culminating in the article of
Buzzi--Maume \cite{BuMa} where the existence of physical
measures (or more generally equilibrium measures) was proved
under very weak additional assumptions.

For piecewise hyperbolic maps, finding good functional spaces
on which the transfer operator has a spectral gap is a more
complicated task, and the story went in the other direction,
with the elementary (but very involved) arguments of  Chernov and Young \cite{Ch0,
Yo,
Ch} coming first. Indeed, even for {\it smooth} hyperbolic
dynamics, good spaces of distributions were only introduced
a few years ago by Gou\"{e}zel--Liverani and Baladi--Tsujii
\cite{GL1, BT1, GL2, BT2}, following the pioneering
work of Blank--Keller--Liverani \cite{BKL}. These spaces cannot be used for piecewise
hyperbolic systems because they are not invariant under multiplication
by the characteristic function of a set with smooth boundary.
Only very recently, a good functional space was
constructed by Demers and Liverani \cite{DL}, for two-dimensional piecewise
hyperbolic maps. However,
the arguments in this last paper are close in spirit to the
previous ones \cite{Yo,
Ch}, in the sense that pieces of stable or unstable
manifolds are iterated by the dynamics, and the way they are
cut by the discontinuities has to be studied in a very careful
way. In particular, to ensure
sufficiently precise control, an essential assumption in \cite{Yo,
Ch, DL} is transversality between stable or unstable
manifolds and discontinuity hypersurfaces.

In this paper, we show that, under mild additional assumptions, the transfer
operator of piecewise hyperbolic maps in
arbitrary dimensions has a spectral gap
on classical  functional spaces $\HH_p^{t,t_-}$,
for suitable indices $t_-<0<t$ and $1<p<\infty$. These spaces are
anisotropic Sobolev spaces in the Triebel-Lizorkin class
\cite{VoPa, Tr}.
Moreover, we are able to replace
the strong transversality assumption from \cite{Yo,
Ch, DL} with a much weaker one,
formulated in terms of the geometry of stable manifolds and
discontinuity hypersurfaces: for instance, we allow discontinuity sets
coinciding with pieces of stable manifolds. Of course, this
implicitly assumes the existence of stable
manifolds, and this may be the main current restriction of our approach: we
require stable manifolds to exist everywhere, and to depend in
a piecewise $C^{1+\alpha}$ way on the point for some $\alpha>0$.
(See also Remarks~ \ref{rmk:Eslisse} and ~\ref{remark7}.)

The main novelty in this work is that, as in \cite{Sau, Co},
we do not need to study precisely the dynamics. In particular,
we do not iterate single stable or unstable manifolds (contrary
to \cite{Yo, Ch, DL}), and we do not need to match nearby stable
or unstable manifolds. Indeed, everything comes from the
functional analytic framework. This makes it possible to get a
short self-contained proof working in any dimension and with
very weak transversality assumptions.

Our spaces $\HH_p^{t,t_-}$ (or more precisely their $\tilde
\HH_p^{t_+,t}$ version, see Remark~ \ref{remark7}) are the same
the first named author considered in \cite{baladi:Cinfty} (with
the notation $W^{t,t_+,p}$) to study smooth hyperbolic maps.
The main new observation that we shall use is (Lemma~\ref{lem:multiplier}) that these spaces
are stable under multiplication by characteristic functions of
nice sets, if the smoothness indices in the definition of the
space are small enough with respect to the integrability index
($0<t<1/p$ and $0>t_- >-1+1/p$). This property is well known
(see the thesis \cite{Str} of Strichartz, and also \cite[\S
4.6.3]{RS}) for classical Sobolev spaces where $t_-=0$, and we
will exploit some ideas in \cite{Str} to prove that it extends
to our spaces. For this, we use complex interpolation arguments
to extend easily to our spaces estimates that are
straightforward for the standard Sobolev spaces. Interpolation
also makes it possible to generalize the basic estimates in
\cite{baladi:Cinfty} to arbitrary differentiability  (see
Appendix~\ref{whynotbaladi}). Another helpful technical
ingredient is the use of a ``zooming" norm (\ref{zoom}) (based
on a rather standard localization principle) which allows us to
go further than  \cite{baladi:Cinfty}, which only dealt with
specific transfer operators.

\smallskip
We do not believe that our upper bounds on the essential
spectral radius are also lower bounds in general. However, we
note that for a (non necessarily Markov) piecewise linear map
of the unit square given by a hyperbolic matrix $A$  of maximal
eigenvalue $\lambda >1$ (see Subsection~\ref{pwaff}), we find
for each $\epsilon >0$ a space on which the essential spectral
radius of the ordinary (Perron-Frobenius) transfer operator is
$\le \lambda^{-1/2+\epsilon}$. This is sharper than the results
in \cite{DL}, and may well be the optimal bound  (in the strong
sense of meromorphic extensions of the corresponding zeta
function or essential decorrelation rate \cite{CE}). We refer
also to Subsection~\ref{pwaff} for examples of conservative and
dissipative (sloppy) baker maps to which our results apply.

\smallskip
Our proof extends the results of \cite{baladi:Cinfty} to
$C^{1+\alpha}$ Anosov diffeomorphisms with $C^{1+\alpha}$
stable and/or unstable distributions, and general $C^\alpha$
weights (see Remark~ \ref{extend}). Let us also mention that
our results  apply to piecewise expanding  and piecewise
$C^{1+\alpha}$ maps for $0<\alpha<1$ (without  transversality
assumptions, but under the hypothesis that the dynamical
complexity does not grow too fast), giving yet another
functional space on which the results of Saussol and Cowieson
\cite{Sau, Co}, e.g., hold.
This space is simply the usual Sobolev space $\HH^t_p$ for $1<p
<\infty$ and $0<t <\min(1/p,\alpha)$. Hence, introducing exotic
spaces to study piecewise expanding maps is not necessary. This
remark seems to be new even for one-dimensional piecewise
expanding  maps. (For {\it smooth} expanding maps in arbitrary
dimensions, the transfer operator was studied on Sobolev spaces
in \cite{BB}.)

\smallskip
The paper is organized as follows. Section 1 contains the
definitions (Definition~ \ref{def:PiecHyp}) of the dynamics $T$
considered (in particular, the condition on the stable
foliation) and the spaces (Definition~ \ref{def:spaceX})
$\HH_p^{t,t_-}$, as well as our weak transversality condition
(Definition~ \ref{def:WTC}), and our main result. This main
result, Theorem ~ \ref{thm:MainSpectralThm}, gives a bound on
the essential spectral radius of the transfer operator acting
on $\HH_p^{t,t_-}$. We give  in Corollary~ \ref{cor:BoundSRB}
the consequences of our main result on the existence of
finitely many physical measures with total ergodic basin (based
on a key result given in Appendix~ \ref{sec:SRB}), as well as
variants of this main result under  assumptions on the unstable
foliation. Section~ 2 is devoted to a discussion of several
examples, illustrating our conditions. In Section~ 3, we recall
various classical results in functional analysis. Section~ 4 is
the heart of the paper: it contains the basic bounds
(multiplication by a function, composition by a smooth map
preserving the stable foliation, multiplication by the
characteristic function of a nice set) which lead to
Lasota-Yorke type inequalities. In Section~ 5, we exploit these
bounds, using a new ``zooming" trick made possible by the
localization property of our spaces, to prove Theorem ~
\ref{thm:MainSpectralThm}.

%%%%%%%%%%%%%%%%%%%%%%%%%%%%%%%%%%%%%%%%%%%%%%%%%%%%%%%%%%%%%%%%%%%%

\section{Statements}

Notations: if $B$  is a Banach
space, we denote the norm of an element $f$ of $B$ by
$\norm{f}{B}$. In this paper, a function defined on a closed
subset of a manifold is said to be $C^k$ or $C^\infty$ if it
admits an extension to a neighborhood of this closed subset,
which is $C^k$ or $C^\infty$ in the usual sense.
%This is the differentiability in the sense of Whitney.

%%%%%%%%%%%%%%%%%%%%%%%%%%%%%%%%%%%%%%%%%

\subsection{The setting}
Let $X$ be a riemannian manifold of dimension $d$, and let
$X_0$ be a compact subset of $X$. Let also $0\le d_s\leq d$ and
$\alpha> 0$. We call $C^1$ hypersurface with boundary a
codimension-one $C^1$ submanifold of $X$ with boundary. For a
closed subset $K$ of $X_0$ we shall consider {\it integrable
$C^{1+\alpha}$ distributions} of $d_s$-dimensional subspaces
$E^s$ on $K$. By definition, this means that for each $x$ in a
neighborhood of $K$, $E^s(x)$ is a $d_s$-dimensional vector
subspace of the tangent space $\T_x X$, the map $x\mapsto
E^s(x)$ is $C^{1+\alpha}$ and, for any $x\in K$, there exists a
unique submanifold of dimension $d_s$ containing $x$, defined
on a neighborhood of $x$, and everywhere tangent to $E^s$. We
will denote this local submanifold by $W^s_{loc}(x)$, and by
$W^s_\epsilon(x)$ we will mean the ball of size $\epsilon$
around $x$ in this submanifold.

\begin{definition}[Piecewise hyperbolic maps with stable distribution]
\label{def:PiecHyp} For $\alpha>0$, we say that a map $T: X_0
\to X_0$ is a piecewise $C^{1+\alpha}$ hyperbolic map with
smooth stable distribution if
\begin{itemize}
\item There exists an integrable $C^{1+\alpha}$ distribution of $d_s$-dimensional
subspaces $E^s$ on a neighborhood of $X_0$.
\item There exists a finite number of disjoint open subsets
$O_1,\dots,O_I$ of $X_0$, covering Lebesgue-almost all
$X_0$, whose boundaries are unions of finitely many compact
$C^1$ hypersurfaces with boundary.
\item For $1\leq i\leq I$, there exists a $C^{1+\alpha}$ map $T_i$
defined on a neighborhood of $\overline{O_i}$, which is a
diffeomorphism onto its image, such that $T$ coincides with
$T_i$ on $O_i$.
\item For any $x\in \overline{O_i}$, there exists $\lambda_s(x)<1$
such that, for any $v\in E^s(x)$, $DT_i(x)v\in E^s(T_i x) $
and $|DT_i(x) v| \leq \lambda_s(x) |v|$.
\item There exists a family of cones $C^u(x)$, depending continuously
on $x\in X_0$, with $C^u(x)+E^s(x)=\T_x X$, such that, for
any $x\in \overline{O_i}$, $DT_i(x) C^u(x) \subset C^u(T_i
x)$, and there exists $\lambda_u(x)>1$ such that $|DT_i (x)
v|\geq \lambda_u(x) |v|$ for any $v\in C^u(x)$.
\end{itemize}
\end{definition}

See Remark~\ref{remark7} and Subsection~\ref{exchangeus}
regarding the replacement of $E^s$ by $E^u$ and $C^u$ by $C^s$
in the above definition.

Note that we do not assume that $T$ is continuous or injective
on $X_0$.

When $d_s=0$, the map $T$ is piecewise expanding. When $d_u=0$,
it is piecewise contracting (we shall see that our results are
not very useful in this case). In the intermediate case, there
are at the same time contracted and expanded directions. We
will denote by $\lambda_{s,n}(x)<1$ and $\lambda_{u,n}(x)>1$
the weakest contraction and expansion constants of $T^n$ at
$x$.

\begin{rmk}
\label{rmk:Eslisse} The requirement that $E^s$ is defined
everywhere and $C^{1+\alpha}$ is extremely strong. It is
possible to weaken it slightly, by requiring only that $E^s$ is
$C^{1+\alpha}$ on each set $O_i$. Indeed, our proofs still work
under this weaker assumption (one should just slightly modify
the definition of the Banach space we use). It is also possible
to apply directly our results to this more general setting, by
working on a different manifold, as follows. Assume that $T$ is
a piecewise hyperbolic map for which $E^s$ is $C^{1+\alpha}$ on
each set $O_i$, but not globally. Start from the disjoint union
of the sets $\overline{O_i}$, and glue them together at all the
points $x\in \overline{O_i} \cap \overline{O_j}$ such that
$E^s$ is $C^{1+\alpha}$ on a neighborhood of $x$. Then $T$
induces a piecewise hyperbolic map on this new manifold, for
which the stable distribution is globally $C^{1+\alpha}$.
Indeed, since $T$ is $C^{1+\alpha}$ on each set $O_i$, the set
$T(O_i)$ intersects the boundaries of the sets $O_j$ only at
places where $E^s$ is $C^{1+\alpha}$. Hence, the places in the
original manifold where $\overline{O_i}$ and $\overline{O_j}$
are cut apart are not an obstruction to extending $T$ to the
new manifold. The assumption on the $C^u$ can be similarly
weakened.
\end{rmk}

In order to define our weak transversality
condition on the boundaries of the sets $O_i$, we shall use the following notion.

\begin{definition}[$L$-generic vector in $E^s$]
Let $K\subset X_0$ be a compact hypersurface with boundary and
let $L  \in \integer_+$. For $x\in K\backslash \partial K$, we
say that a vector $a\in E^s(x)$ is \emph{$L$-generic} with
respect to $K$ if, for any $C^1$ vector field $v$ defined on a
neighborhood of $x$, with $v(x)=a$ and $v(y)\in E^s(y)$ for any
$y$, there exists a smaller neighborhood of $x$ in which the
intersection of Lebesgue almost every integral line of $v$ with
$K$ has at most $L$ points.
\end{definition}

\begin{definition}[Weak transversality condition for $E^s$]
\label{def:WTC} Let $T:X_0 \to X_0$ be a piecewise hyperbolic
map with smooth stable distribution. We say that $T$ satisfies
the \emph{weak transversality condition} if there exists $L>0$
such that, for any $K \subset \bigcup_{i=1}^I \partial O_i$
which is hypersurface with boundary, there exists a larger
hypersurface with boundary $K'$ (containing $K$ in its
interior) such that, for any $x\in K'\backslash \partial K'$,
the set of tangent vectors at $x$ that are $L$-generic with
respect to $K'$ has full Lebesgue measure in
$E^s(x)$.\footnote{We could replace ``full Lebesgue measure''
in this definition by ``generic in the sense of Baire'' (i.e.,
contains a countable intersection of dense open sets), all the
following results would hold true as well, with the same
proofs.}
\end{definition}
The small enlargement $K'$ of $K$ is simply a technical point
in the definition, to avoid problems at the boundary of $K$.

If the boundary of each $O_i$ is a finite union
of smooth hypersurfaces $K_{i1},\dots, K_{ik_i}$, each of which
is transversal to the stable direction (in the sense that
$E^s(x)$ is never contained in $\T_x K_{ij}$), then $T$
satisfies the weak transversality condition. However, the
converse does not hold. For instance, we have the
following result:

\begin{prop}
Assume that $d_s=1$ (so that the stable manifolds are curves),
and that $T$ is a piecewise hyperbolic map with smooth stable
distribution. Then $T$ satisfies the weak transversality
condition if there exists $\epsilon>0$ such that
  \begin{equation}
  \sup_{1\leq i \leq I} \norm{
  \Card( W^s_\epsilon(x)\cap \partial O_i)}{L_\infty(\Leb)} < \infty.
  \end{equation}
\end{prop}
Hence, tangencies to the boundaries of the $O_i$'s are allowed,
and even flat tangencies or pieces of the boundary coinciding
with $W^s$. The only problematic situation is when a boundary
oscillates around the stable manifold, cutting it into
infinitely many small pieces.
% as in the following pathological
% case: assume that $d=2$, that $E^s$ is the vertical direction,
% and that $O=\{(x,y) \st x > y^2 \sin(1/y) \}$.

\smallskip

To get a result on the physical measures of finitely
differentiable maps $T$, it is necessary to add {\it some}
assumption on the asymptotic dynamical complexity, already for
piecewise expanding maps in dimension two or higher (see
\cite{Sau}, \cite{BuMa}, \cite{Co}, \cite{Ts00} and
\cite{Bu01}). We shall use the following way to quantify the
complexity.

Let $\ii=(i_0,\dots,i_{n-1})\in \{1,\dots,I\}^n$. We define
inductively sets $O_\ii$ by $O_{(i_0)}=O_i$, and
  \begin{equation}
  O_{(i_0,\dots,i_{n-1})}=\{x\in O_{i_0} \st T_{i_0}x\in
  O_{(i_1,\dots,i_{n-1})}\}.
  \end{equation}
Let also $T_\ii=T_{i_{n-1}}\circ \dots \circ T_{i_0}$, it is
defined on a neighborhood of $O_\ii$.

We define the complexity at the beginning
  \begin{equation}\label{cpb}
  D^b_n=\max_{x\in X_0} \Card \{ \ii=(i_0,\dots,i_{n-1}) \st x \in
  \overline{O_{\ii}} \},
  \end{equation}
and the complexity at the end
  \begin{equation}\label{cpe}
  D^e_n=\max_{x\in X_0} \Card \{ \ii=(i_0,\dots,i_{n-1}) \st x \in
  \overline{T^n(O_{\ii})} \}.
  \end{equation}

%%%%%%%%%%%%%%%%%%%%%%%%%%%%%%%%%%%%%%%%%%%%%%%%%%%

\subsection{The main spectral result}

We shall use spaces $\HH_p^{t,t_-}$ which were first introduced in a dynamical
setting in \cite{baladi:Cinfty} (the local version of these
spaces belongs to the Triebel-Lizorkin class, see \cite{VoPa},
\cite{Bagh}, \cite{Tr} for earlier mentions of these  spaces in
functional analysis). Section \ref{sec:local} is devoted to a
precise study of these spaces, and the statements in the
following definition are justified there.

Let $\FF$ denote the Fourier transform in $\R^d$. We will write
a point $z\in \R^d$ as $z=(x,y)$ where $x=(z_1,\dots,z_{d_u})$
and $y=(z_{d_u+1},\dots, z_d)$. In the same way, an element
$\zeta$ of the dual space of $\R^d$ will be written as
$\zeta=(\xi,\eta)$. The subspaces $\{x\}\times \R^{d_s}$ of
$\R^d$ will sometimes be referred to as the ``stable leaves''
in $\R^d$. We say that a diffeomorphisms sends stable leaves to stable
leaves if its derivative has this property.

\begin{definition}[Local spaces $H^{t,t_-}_p$]
\label{def:space} For $1<p<\infty$, $t, t_- \in \real$,
we define a space $H_p^{t,t_-}$ of distributions in $\R^d$ as the
(tempered) distributions $u$ such that
$$\FF^{-1}(
(1+|\xi|^2+|\eta|^2)^{t/2} (1+|\eta|^2)^{t_-/2} \FF u)\in L_p,
$$
 with its canonical norm.
\end{definition}

We will simply write
$H_p^t$ instead of $H_p^{t,0}$.

If $t \ge 0$, $t+t_- \le 0$
and $t+|t_-|<\alpha<1$, we shall see that $H^{t,t_-}_p$
is invariant under $C^{1+\alpha}$ diffeomorphisms
sending stable leaves to stable leaves
(Remark~\ref{rmk:Invariance}). Hence, we can glue such
spaces locally together in appropriate coordinate patches, to
define a space $\HH_p^{t,t_-}$ of distributions on the manifold:

\begin{definition}[Spaces  $\HH_p^{t,t_-}$ of distributions on $X$]
\label{def:spaceX} Let $t \ge 0$, $t+t_- \le 0$ and
$t+|t_-|<\alpha<1$. Fix  a finite number of $C^{1+\alpha}$
charts $\kappa_1,\dots,\kappa_J$ whose derivatives send $E^s$
to $\{0\}\times \R^{d_s}$, and whose domains of definition
cover a compact neighborhood of $X_0$, and a partition of unity
$\rho_1,\dots,\rho_J$, such that the support of $\rho_j$ is
compactly contained in the domain of definition of $\kappa_j$,
and $\sum \rho_j=1$ on $X_0$. The space $\HH_p^{t,t_-}$ is then
the space of distributions\footnote{\label{page:foot}On a manifold, the
space of \emph{generalized functions} supported in $X_0$, i.e.,
elements in the dual of the space of smooth densities, and the
space of \emph{generalized densities} supported in $X_0$, i.e.,
elements in the dual of the space of smooth functions, are
isomorphic if $X_0$ is compact: taking $\Leb$ any smooth
riemannian measure then $f\mapsto f\dLeb$ gives an isomorphism.
``Distributions supported in $X_0$" (not to be confused with
the integrable distributions of subspaces in
Definition~\ref{def:PiecHyp}) refers in this paper to
generalized functions (this avoids jacobians in the change
of variables).} $u$ supported on $X_0$  such that
$(\rho_j u)\circ \kappa_j^{-1}$ belongs to $H_p^{t,t_-}$ for
all $j$, endowed with the norm
  \begin{equation}
  \label{defnorm}
  \norm{u}{\HH_p^{t,t_-}}=\sum \norm{ (\rho_j u) \circ
  \kappa_j^{-1}}{H_p^{t,t_-}}.
  \end{equation}
\end{definition}

Changing the charts and the partition of unity gives
an equivalent norm on the same space of distributions
by Lemma~\ref{Leib} and Remark~\ref{rmk:Invariance}.
To fix ideas, we shall view the charts and partition of unity as fixed.

\begin{rmk}\label{remark7}
Note that \cite{baladi:Cinfty} considers a slightly different
space, where the stable and unstable direction and the signs of
$t$ and $t+t_-$ are exchanged. This choice is completely
innocent, we also get the same results for the space of
\cite{baladi:Cinfty} (for maps with smooth unstable
distribution) in Theorem \ref{thm:smooth_unstable}.
\end{rmk}

Our main result follows (recall the notation
\eqref{cpb}--\eqref{cpe}):

\begin{thm}[Spectral theorem for smooth stable distributions]
\label{thm:MainSpectralThm} Let $\alpha\in (0,1]$. Let $T$ be a
piecewise $C^{1+\alpha} $hyperbolic map with smooth stable
distribution, satisfying the weak transversality condition. Let
$1<p<\infty$ and let $t,t_-$ be so that $1/p-1<t_-<0<t<1/p$,
$t+t_-<0$ and $t+|t_-|<\alpha$.

Let $g:X_0\to \C$ be a function such that the restriction of
$g$ to any $O_i$ admits a $C^{\alpha}$ extension to
$\overline{O_i}$. Define an operator $\Lp_g$ acting on bounded
functions by $(\Lp_g u)(x)=\sum_{Ty=x} g(y)u(y)$. Then $\Lp_g$
acts continuously on $\HH_p^{t,t_-}$. Moreover, its essential
spectral radius is at most
  \begin{equation}\label{bdess}
  \lim_{n\to\infty}
  (D_n^b)^{1/(pn)} \cdot
  (D_n^e)^{(1/n)(1-1/p)}
  \cdot \norm{g^{(n)}|\det DT^n|^{1/p}\max(\lambda_{u,n}^{-t},
  \lambda_{s,n}^{-(t+t_-)})}{L_\infty}^{1/n},
  \end{equation}
where $g^{(n)}=\prod_{j=0}^{n-1}g\circ T^j$.
\end{thm}

When we say that $\Lp_g$ acts continuously on $\HH_p^{t,t_-}$,
we should be more precise. We mean that, for any $u\in
\HH_p^{t,t_-}\cap L_\infty(\Leb)$, then $\Lp_g u$, which is
defined as a bounded function, still belongs to $\HH_p^{t,t_-}$
and satisfies $\norm{\Lp_g u}{\HH_p^{t,t_-}}\leq C
\norm{u}{\HH_p^{t,t_-}}$. Since the set of bounded functions is
dense in $\HH_p^{t,t_-}$ (by Lemma~\ref{14}), the operator
$\Lp_g$ can therefore be extended to a continuous operator on
$\HH_p^{t,t_-}$.

Note that the limit in (\ref{bdess}) exists by
submultiplicativity. Of course, we can bound $\lambda_{s,n}$
and $\lambda_{u,n}^{-1}$ by $\lambda^n$, where $\lambda<1$ is
the weakest rate of contraction/expansion of $T$. In some
cases, it will be important to use the more precise expression
given above (see e.g.~Example \ref{ex5} below).

The restriction $1/p-1<t_-<0<t<1/p$ is exactly designed so that
the space $\HH_p^{t,t_-}$ is stable under multiplication by
characteristic functions of nice sets, see Lemma
\ref{lem:multiplier}. While this feature will be used in an
essential way in the proof, it also implies (see Remark~\ref{nodirac}
in Appendix~\ref{sec:SRB}) that Dirac measures
(or more generally measures supported on nice hypersurfaces) do
not belong to the space $\HH_p^{t,t_-}$.

%%%%%%%%%%%%%%%%%%%%%%%%%%%%%%%%%%%%%%%%%%%%%%%%%%%%%%%%%%%
\subsection{Physical measures}

The physical measures of $T$ are by definition the
probability measures $\mu$ such that there exists a set $A$ of
positive Lebesgue measure such that, for all $x\in A$, $1/n
\sum_{k=0}^{n-1}\delta_{T^k x}$ converges weakly to $\mu$.

The physical measures of $T$ are often studied
through the transfer operator $\LSRB$.
(Note that the dual of $\LSRB$ preserves Lebesgue measure.)
Theorem \ref{thm:MainSpectralThm} becomes in this setting:
\begin{cor}
\label{cor:BoundSRB} Under the assumptions of Theorem
\ref{thm:MainSpectralThm}, assume that
  \begin{equation}
  \label{eq:qlksdjfml}
  \lim_{n\to\infty}
  (D_n^b)^{1/(np)}\cdot
  (D_n^e)^{(1/n)(1-1/p)} \cdot \norm{\max(\lambda_{u,n}^{-t},
  \lambda_{s,n}^{-(t+t_-)})|\det
  DT^n|^{1/p-1}}{L_\infty}^{1/n} <1.
  \end{equation}
Then the essential spectral radius of $\LSRB$ acting on
$\HH_p^{t,t_-}$ is $<1$.
\end{cor}

Together with classical arguments, this implies the following:
\begin{thm}
\label{thm:ExistSRB} Under the assumptions of Theorem
\ref{thm:MainSpectralThm}, if \eqref{eq:qlksdjfml} holds, then $T$
has a finite number of physical measures, which are invariant and
ergodic, whose basins cover Lebesgue almost all $X_0$.
Moreover, if $\mu$ is one of these measures, there exist an
integer $k$ and a decomposition $\mu=\mu_1+\dots+\mu_k$ such
that $T$ sends $\mu_j$ to $\mu_{j+1}$ for $j\in \Z/k\Z$, and
the probability measures $k\mu_j$ are exponentially mixing for
$T^k$ and H\"{o}lder test functions.
\end{thm}
The deduction of this theorem from Corollary \ref{cor:BoundSRB}
is essentially folklore, but the proofs of similar results in
the literature (e.g.~in \cite{BKL,DL}) rely on some properties
of stable manifolds that are not established in our setting. We
prove in Appendix ~ \ref{sec:SRB} a  general
theorem (Theorem~\ref{thm:SRBabstrait})
that guarantees the existence of finitely many physical
measures whenever the transfer operator has a spectral gap on a
space of distributions, and show (Lemma~\ref{deduce})
that this general theorem holds in our setting. The interest of this argument is that
it also applies to non hyperbolic situations, such as
(perturbations of the operators in) \cite{Tsujii}.

The results in this subsection answer the question in \cite[Remark 1.1]{baladi:Cinfty},
in a much more general framework.

\subsection{Hyperbolic maps with smooth unstable distribution}
\label{exchangeus}
Just like in Definition \ref{def:PiecHyp}, we can define
piecewise $C^{1+\alpha}$ hyperbolic maps with smooth unstable
distribution. Our results also apply to such maps (by the same
techniques used to prove Theorem \ref{thm:MainSpectralThm}),
but on the space of distributions $\tilde \HH^{t_+,t}$ whose
norm is given in charts by $\norm{\FF^{-1} ((1+|\xi|^2)^{t_+/2}
(1+|\xi|^2+|\eta|^2)^{t/2} \FF u)}{L_p}$. More precisely:

\begin{thm}[Spectral theorem for smooth unstable distributions]
\label{thm:smooth_unstable} Let $\alpha\in (0,1]$. Let $T$ be a
piecewise $C^{1+\alpha} $hyperbolic map with smooth unstable
distribution, satisfying the weak transversality condition with
$E^s$ replaced by $E^u$. Let
$1<p<\infty$ and let $t_+$, $t$ be so that $1/p-1<t<0<t_+<1/p$,
$t+t_+>0$ and $|t|+t_+<\alpha$.

Let $g:X_0\to \C$ be a function such that the restriction of
$g$ to any $O_i$ admits a $C^{\alpha}$ extension to
$\overline{O_i}$. Define an operator $\Lp_g$ acting on bounded
functions by $(\Lp_g u)(x)=\sum_{Ty=x} g(y)u(y)$. Then $\Lp_g$
acts continuously on $\tilde\HH_p^{t_+,t}$. Moreover, its
essential spectral radius is at most
  \begin{equation*}
  \lim_{n\to\infty}
  (D_n^b)^{1/(pn)} \cdot
  (D_n^e)^{(1/n)(1-1/p)}
  \cdot \norm{g^{(n)}|\det DT^n|^{1/p}\max(\lambda_{u,n}^{-(t+t_+)},
  \lambda_{s,n}^{-t})}{L_\infty}^{1/n}.
  \end{equation*}

In particular, if
  \begin{equation*}
  \lim_{n\to\infty}
  (D_n^b)^{1/(np)}\cdot
  (D_n^e)^{(1/n)(1-1/p)} \cdot \norm{\max(\lambda_{u,n}^{-(t+t_+)},
  \lambda_{s,n}^{-t})|\det
  DT^n|^{1/p-1}}{L_\infty}^{1/n} <1,
  \end{equation*}
then the spectral radius of $\LSRB$ acting on
$\tilde\HH_p^{t_+,t}$ is $<1$. This implies that $T$ has a
finite number of ergodic physical measures whose basins cover
Lebesgue almost all $X_0$. Moreover, if $\mu$ is one of these
measures, there exist an integer $k$ and a decomposition
$\mu=\mu_1+\dots+\mu_k$ such that $T$ sends $\mu_j$ to
$\mu_{j+1}$ for $j\in \Z/k\Z$, and the probability measures
$k\mu_j$ are exponentially mixing for $T^k$ and H\"{o}lder test functions.
\end{thm}

We will not give further details on the proof of this theorem,
since it follows from the techniques used in the proof of
Theorem \ref{thm:MainSpectralThm}.

Finally, similar results hold for maps that have at the same
time smooth stable and unstable distributions (and satisfy the
weak transversality condition in both directions), as follows.
Let $\tilde{\tilde{\HH}}_p^{t_+,t_-}$ be the space of
distributions whose norm is given in charts by $\norm{\FF^{-1}
((1+|\xi|^2)^{t_+/2} (1+|\eta|^2)^{t_-/2} \FF u)}{L_p}$.

\begin{thm}[Spectral theorem when both distributions are smooth]\label{both}
Let $T$ be a piecewise $C^{1+\alpha}$
hyperbolic map with smooth stable and unstable distribution,
satisfying the weak transversality conditions for $E^s$ and $E^u$
for $\alpha\in (0,1]$. Let $1<p<\infty$
and let $t_+$, $t_-$ be so that $1/p-1<t_-<0<t_+<1/p$, and
$|t_-|+t_+<\alpha$.

Let $g:X_0\to \C$ be a function such that the restriction of
$g$ to any $O_i$ admits a $C^{\alpha}$ extension to
$\overline{O_i}$. Define an operator $\Lp_g$ acting on bounded
functions by $(\Lp_g u)(x)=\sum_{Ty=x} g(y)u(y)$. Then $\Lp_g$
acts continuously on $\tilde{\tilde{\HH}}_p^{t_+,t_-}$.
Moreover, its essential spectral radius is at most
  \begin{equation}\label{thebest}
  \lim_{n\to\infty}
  (D_n^b)^{1/(pn)} \cdot
  (D_n^e)^{(1/n)(1-1/p)}
  \cdot \norm{g^{(n)}|\det DT^n|^{1/p}\max(\lambda_{u,n}^{-t_+},
  \lambda_{s,n}^{-t_-})}{L_\infty}^{1/n}.
  \end{equation}
\end{thm}
The results on physical measures follow analogously. It should
be noted that the results of Theorem~ \ref{both} are stronger
than Theorems~ \ref{thm:MainSpectralThm} and~
\ref{thm:smooth_unstable}, since the exponents $t_+$ and $t_-$
appear independently in the estimate \eqref{thebest}.

Once again, this theorem follows from the techniques we will
use to prove Theorem \ref{thm:MainSpectralThm}.

%%%%%%%%%%%%%%%%%%%%%%%%%%%%%%%%%%%%%%%%%%%%%%%

\section{Examples}
\label{sec:examples}

Let us look at some applications of our results to $\LSRB$.

\subsection{General examples}

\begin{ex}
\label{ex2} On $[-1,1]\times \{0,1\}$, let $T(x,j)=(x/2,j)$ if
$x\not=0$, and $T(0,j)=(0,1-j)$. This fits in our framework.
Since the complexities $D^b_n$ and $D^e_n$ are always equal to
$2$, Theorem~ \ref{thm:MainSpectralThm} gives the following
bound for the essential spectral radius of $\LSRB$ on the
classical Sobolev space $\HH_p^{t_-}$:
  \begin{equation}
  \lim_{n \to \infty} \norm{\lambda_{s,n}^{-t_-}|\det DT^n|^{1/p-1}}{L_\infty}^{1/n}
  =2^{t_-+1-1/p}.
  \end{equation}
Since $t_-<0$ is restricted by $t_->1/p-1$, this bound is $>1$,
hence useless. This is not surprising since the physical
measure, the Dirac mass at $0$, does not belong to
$\HH_p^{t_-}$ if $1/p-1<t_-<0$ (see Remark \ref{nodirac}).

This was to be expected since the conclusion of
Theorem~\ref{thm:ExistSRB} is false: the map $T$ has two
physical measures, the Dirac masses at $(0,0)$ and $(0,1)$, but
these measures are not invariant!

It is nevertheless interesting to see where precisely our
arguments fail. Let $\tilde T(x,j)=(x/2,j)$, then the transfer
operators associated to $T$ and $\tilde T$ acting on
distributions coincide on $C^\infty$ functions (since the
difference at $0$ is not seen by the integration against smooth
functions). Since $\tilde T$ is continuous, there is no
truncation term in its transfer operator, hence the results of
Theorem \ref{thm:MainSpectralThm} hold for the full range
$t_-<0$, without the restriction $t_->1/p-1$ (with the same
proof). In particular, for $t_-=-1$ and $p=2$, we get a bound
$1/\sqrt{2}$ for the essential spectral radius of $\Lp_{1/\det
DT}(T)=\Lp_{1/\det D\tilde T}(\tilde T)$ acting on
$\HH_2^{-1}$, and Corollary \ref{cor:BoundSRB} holds. The
problem comes up in the deduction of the properties of physical
measures from this bound on the essential spectral radius of
$\LSRB$: we need to check that the physical measures do not
give weight to the discontinuities of the map, to apply Theorem
\ref{thm:SRBabstrait}. This is ensured by Lemma \ref{deduce}
when $t_->1/p-1$, but does not hold for $t_-=-1$ and $p=2$.
\end{ex}

\begin{ex}
Assume that $d_s=0$, i.e., $T$ is piecewise expanding. In this
case, we can take $\lambda_s=0$, and the value of $t_-$ is
irrelevant (in fact, the space $\HH_p^{t,t_-}$ does not depend on
$t_-$, and is the classical Sobolev space $\HH_p^t$).

\begin{prop*}
If $T$ is piecewise $C^{2}$, if $d_s=0$ and $\lim
\norm{\lambda_{u,n}^{-1}}{L_\infty}^{1/n} \cdot \lim
(D_n^b)^{1/n} <1$, then there exist $0<t<1/p<1$ such that the
spectral radius of $\LSRB$ acting on $\HH_p^t$ is $<1$. In
particular, Theorem \ref{thm:ExistSRB} applies.
\end{prop*}
\begin{proof}
When $\epsilon$ tends to $0$, the bound on the essential
spectral radius of $\LSRB$ acting on
$\HH_{(1-\epsilon)^{-1}}^{1-2\epsilon}$, given by Corollary
\ref{cor:BoundSRB}, converges at most to $\lim_{n \to \infty}
\norm{\lambda_{u,n}^{-1}}{L_\infty}^{1/n} \cdot \lim_{n \to \infty}
(D_n^b)^{1/n}$. Hence, it is $<1$ for small enough $\epsilon$.
\end{proof}
\end{ex}
In the proof of the above proposition, we use parameters $t$ and $p$ very close to $1$,
but
we are ``morally" working with $\HH_1^1$. This is not surprising
since this space is essentially a space of functions with one
derivative in $L_1$, i.e., a space of functions of bounded
variation. It is well known that functions of bounded variation
are useful to study piecewise expanding maps, see \cite{Co}.
This proposition is analogous to results proved in
\cite{Sau,Co} for different Banach spaces.

\begin{ex}
\label{ex5} When $\det DT=1$ and $D^e_n$, $D^b_n$ grow
subexponentially fast, then it is clear from Corollary
\ref{cor:BoundSRB} that the essential spectral radius of
$\LSRB$ is $<1$ on any space $\HH_p^{t,t_-}$ (as soon as $t>0$
and $t+t_-<0$). In some situations, it is possible to weaken
(or even remove) the assumption that $\det DT=1$. We get more
precise results using Theorem \ref{thm:smooth_unstable}, i.e.,
assuming that the unstable direction is smooth.

\begin{prop*}
Let $T$ be a piecewise $C^{2}$ hyperbolic map with smooth
unstable distribution satisfying the weak transversality
condition, and such that $D^e_n$ and $D^b_n$ grow
subexponentially. Assume that there exist $N>0$ and $\gamma<1$
such that $\lambda_{s,N}\leq \gamma |\det DT^N|$. Then there
exist $p\in (1,\infty)$ and $1/p-1<t<0<t_+<1/p$ such that the
essential spectral radius of $\LSRB$ acting on
$\tilde\HH_p^{t_+,t}$ is $<1$. In particular, $T$ has finitely
many physical measures whose basins contain  Lebesgue almost
every point.
\end{prop*}
The assumption $\lambda_{s,N}\leq \gamma |\det DT^N|$ is a kind
of pinching condition. It is satisfied whenever $d_s=1$ and
$d_u>0$.

\begin{proof}
We will take $p$ very close to $1$, $t=1/p-1+\epsilon$ and
$t_+=1/p-\epsilon$ for $\epsilon>0$ very small.

We have
  \begin{equation}
  |\det DT^N|^{1/p-1} \lambda_{s,N}^{-t}
  \leq (\gamma^{-1}\lambda_{s,N})^{1/p-1} \lambda_{s,N}^{-(1/p-1)-\epsilon}
  = \gamma^{1-1/p} \lambda_{s,N}^{-\epsilon}.
  \end{equation}
Since $\gamma<1$, this quantity is $<1$ if $\epsilon$ is small
enough (in terms of $p$).

Moreover,
  \begin{equation}
  |\det DT^N|^{1/p-1} \lambda_{u,N}^{-(t_++t)}
  = |\det DT^N|^{1/p-1} \lambda_{u,N}^{1-2/p}.
  \end{equation}
When $p\to 1$, this quantity converges to
$\lambda_{u,N}^{-1}<1$.

Hence, it is possible to choose $p$ and $\epsilon$ such that
  \begin{equation}
  \norm{ |\det DT^N|^{1/p-1} \max( \lambda_{s,N}^{-t},
  \lambda_{u,N}^{-(t+t_+)})}{L_\infty}<1.
  \end{equation}
This concludes the proof.
\end{proof}
\end{ex}

%%%%%%%%%%%%%%%%%%%%%%%%%%%%%%%%%%%%%%%%%%%%%%%%%

\subsection{Piecewise linear maps}
\label{pwaff}

In this paragraph, we describe an explicit class of maps for
which the assumptions of the previous theorems are satisfied.
Let $A$ be a $d\times d$ matrix with no eigenvalue of modulus
$1$. It acts on $\R^d$ in a hyperbolic way, with best
expansion/contraction constants $\lambda_u>1$ and
$\lambda_s<1$. Let $X_0$ be a polyhedral region of $\R^d$, and
define a map $T$ on $X_0$ by cutting it into finitely many
polyhedral subregions $O_1,\dots,O_N$, applying $A$ to each of
them, and then mapping $AO_1,\dots, AO_N$ back into $X_0$ by
translations.

Let $J(n)$ be the covering multiplicity of $T^n$, i.e., the
maximal number of preimages of a point under $T^n$. It is
submultiplicative, hence the limit $J=\lim_{n \to \infty} J(n)^{1/n}$ exists.

\begin{prop}
\label{prop:Affine} The map $T$ is a piecewise hyperbolic map
with smooth stable and unstable distributions (given by the
eigenspaces of $A$ corresponding to eigenvalues of modulus
$<1$, resp.~$>1$). It satisfies the weak transversality
conditions for both stable and unstable distributions.
Moreover, if $J\lambda_s <|\det A|$, there exist $1<p<\infty$,
and $t_+$, $t_-$ so that $1/p-1 <t_- <0 < t_+ < 1/p$ and such
that the essential spectral radius of $\Lp_{1/\det|DT|}$ acting
on $\tilde{\tilde{\HH}}_p^{t_+,t_-}$ is $<1$. Therefore, $T$
satisfies the conclusions of Theorem \ref{thm:ExistSRB}.
\end{prop}

As an example of such a map, one can take
$A=\left(\begin{matrix} 2&1\\1&1
\end{matrix}\right)$. Cutting the torus $\mathbb{T}^2$ into
finitely many squares, applying $A$ to each of these squares,
and then permuting the images of the squares, one obtains a
bijection of the torus (for which $J=1$). Hence,
Proposition~ \ref{prop:Affine} applies. The novelty with respect to previous works
such as \cite{Yo, Ch, DL}
is that the sides of the squares can be taken parallel to the
stable or unstable directions.

\begin{proof}[Proof of Proposition \ref{prop:Affine}]
The weak transversality conditions are  direct consequences of
the definitions.

Let $K$ be the total number of the sides of the polyhedra
$O_i$. Around any point $x$, the boundaries of the sets
$O_{(i_0,\dots,i_{n-1})}$ are preimages of theses sides by one
of the maps $A,\dots,A^{n-1}$, which gives at most $nK$
possible directions. Hence, the claim p.~ 105 in \cite{Bu}
gives $D_n^b \leq 2(nK)^d$. This quantity grows
subexponentially. In the same way, $D_n^e \leq 2 J(n) (nK)^d$.

By Theorem \ref{both}, the essential spectral
radius of $\Lp_{1/\det A}$ acting on
$\tilde{\tilde{\HH}}_p^{t_+,t_-}$ (for suitable values of
$p,t_+,t_-$) is bounded by $J^{1-1/p}|\det A|^{1/p-1}
\max(\lambda_u^{-t_+}, \lambda_s^{-t_-})$. Let us take
$t_+=1/p-\epsilon$, $t_-=1/p-1+\epsilon$ and $p$ close to $1$.
Then $1/p-1<t_-<0<t_+<1/p$, hence Theorem~
\ref{both} applies and yields the following
bound for the essential spectral radius:
  \begin{equation}
  |\det A|^{1/p-1} J^{1-1/p} \max(\lambda_u^{-1/p+\epsilon}, \lambda_s^{1-1/p-\epsilon}).
  \end{equation}
If $p$ is close to $1$ and $\epsilon$ is small enough, this
quantity is $<1$ under the assumptions of the proposition.
(Note that if $\det A=J=1$, choosing $p=1/2$ and
$t_+=1/2-\epsilon$, $t_-=-1/2+\epsilon$ gives better bounds.)
\end{proof}

The standard conservative (piecewise affine) baker's map on the
unit square is given by $T(x,y)= (2x, y/2)$ for $0\le x < 1/2$
and $T(x,y)=(2x-1, (y+1)/2)$ for $1/2\le  x \le 1$.  It fits in
the model of this subsection, for a diagonal matrix $A$ with
eigenvalues $2$ and $1/2$. The baker has an obvious Markov
partition with two pieces, and can thus be analyzed by a
(Lipschitz) symbolic model, which gives an essential
decorrelation rate of $2^{-1/2}$ for Lipschitz observables.
(The physical measure is just Lebesgue measure.) The proof of
the previous proposition gives a bound $2^{-1/2+\epsilon}$  for
the essential spectral radius of $\Lp_{1/\det A}$ on $\tilde
{\tilde \HH}_2^{1/2- \epsilon,-1/2 +\epsilon}$ for arbitrarily
small $\epsilon >0$ (here $J=1$, $\det A=1$, $\lambda_u=2$ and
$\lambda_s=1/2$). For a dissipative baker $T(x,y)= (2x, y/3)$
for $0\le x < 1/2$ and $T(x,y)=(2x-1, (y+2)/3)$ for $1/2\le x
\le 1$ ($\lambda_u=2$ and $\lambda_s=1/3$, $\det A=2/3$ and
$J=1$), the proof of the above proposition  gives a bound $
2^{-1+\epsilon+ (\log 3/\log 6)}$ for the essential spectral
radius on $\tilde {\tilde \HH}_p^{1/p- \epsilon,1/p-1
+\epsilon}$ for $p=\log 6/\log 3$. (Note that the dimension of
the attractor is strictly between $1$ and $2$ in this case.)
The above two examples are piecewise affine hyperbolic maps
with a finite Markov partition. But the following variant, that
we  shall call a ``sloppy baker," does not have a finite Markov
partition:  let $(a,b)$ be a point in the interior of the unit
square and put $T(x,y)= (2x+a, y/2+b) \mod 1$ for $0\le x <
1/2$ and $T(x,y)=(2x-1+a, (y+1)/2+b) \mod 1$ for $0\le x <1$.
For almost all $(a,b)$, the sloppy baker does not have a finite
Markov partition. However, our estimate gives the same bound
$2^{-1/2+\epsilon}$  for the essential spectral radius on
$\tilde {\tilde \HH}_2^{1/2- \epsilon,-1/2 +\epsilon}$.
Similarly, one may consider a dissipative sloppy baker, and we
recover the same estimates.

%%%%%%%%%%%%%%%%%%%%%%%%%%%%%%%%%%%%%%%%%

\section{Tools of functional analysis}

In this section, we  recall some classical notions of functional
analysis (interpolation theory and properties
of Triebel spaces), that  will be useful in the next sections to study
the space $H_p^{t,t_-}$ and to prove our main result.

\subsection{Complex interpolation}

We first recall some notations and definitions from the
classical complex interpolation theory of Lions, Calder\'{o}n and
Krejn (see e.g. \cite{TrB}). A pair $(\BB_0, \BB_1)$ of Banach
spaces is called an interpolation couple if they are both
continuously embedded in a linear Hausdorff space $\BB$. For
any interpolation couple $(\BB_0, \BB_1)$, we let $L(\BB_0,
\BB_1)$ be the space of all linear operators $\LL$ mapping
$\BB_0+\BB_1$ to itself so that $\LL|_{\BB_j}$ is continuous
from $\BB_j$ to itself for $j=0,1$. For an interpolation couple
$(\BB_0, \BB_1)$ and $0 < \theta < 1$, we denote by $[\BB_0,
\BB_1]_\theta$ the complex interpolation space of parameter
$\theta$. We recall the definition: set $S= \{ z\in \complex
\st 0 < \Re z < 1\}$, and introduce the normed vector space
\begin{align*}
F(\BB_0, \BB_1)= \{& f : S \to \BB_0 + \BB_1, \mbox{ analytic,
extending continuously to } \overline S ,\\
\nonumber & \mbox{  with }\sup_{z \in \overline S}
\norm{f(z)}{\BB_0+\BB_1} < \infty,
\mbox{ and }\\
\nonumber &t \mapsto f(j+it)\mbox{ is continuous
from } (-\infty, \infty)\mbox { to }\BB_j , j=0, 1 ,\\
\nonumber &\mbox{and }
 \norm{f}{F(\BB_0,\BB_1)}:=\max_{j=0, 1} ( \sup_t \norm {f(j+it)}{\BB_j}) < \infty \}.
\end{align*}
Then the complex interpolation space is defined for $\theta \in (0,1)$ by
  \begin{equation}
  [\BB_0, \BB_1]_\theta := \{ u \in \BB_0+\BB_1 \st \exists f \in
  F(\BB_0, \BB_1) \mbox{ with } f(\theta)=u\},
  \end{equation}
normed by
  \begin{equation}
  \label{InterpolationNorm}
  \norm{u}{[\BB_0, \BB_1]_\theta}=
  \inf_{f(\theta)=u} \norm{f} {F(\BB_0, \BB_1)}.
  \end{equation}

It is well-known (see e.g. \cite[\S 1.9]{TrB}) that $(\BB_0,
\BB_1)\mapsto [\BB_0, \BB_1]_\theta$ is an exact interpolation
functor of type $\theta$, in the following sense: for any
interpolation couple $(\BB_0, \BB_1)$ and every $\LL \in
L(\BB_0, \BB_1)$ we have
  \begin{equation}\label{interpp}
  \norm{\LL}{[\BB_0, \BB_1]_\theta \to [\BB_0,\BB_1]_\theta} \le
  \norm{\LL}{\BB_0 \to \BB_0}^{1-\theta} \norm{\LL}{\BB_1\to \BB_1}^{\theta} \,
  \quad \forall \, \theta \in (0,1).
  \end{equation}
The above bound will be used several times throughout this
work.

\subsection{A class of Sobolev-like spaces containing the local spaces $H^{t,t_-}_p$}

Let $S$ be the Schwartz space of $C^\infty$ rapidly decaying
functions. Its dual $S'$ is the space of tempered
distributions.

Let $M$ be the set of functions $a$ from $\R^d$ to $\R_+$ such
that there exists $C>0$ such that, for all multi-indices
$\gamma=(\gamma_1,\dots,\gamma_d)$ with $\gamma_j \in \{0,1\}$,
and all $\zeta\in \R^d$,
  \begin{equation}
  \left| \prod_{j=1}^d (1+\zeta_j^2)^{\gamma_j/2} D^\gamma a(\zeta)
  \right| \leq C a(\zeta).
  \end{equation}
For $a\in M$ and $p\in (1,\infty)$, let us define a space
$H_p^a$ as the space of all tempered distributions $u$ such
that $\FF^{-1}( a\FF u)$ belongs to $L_p$, with its canonical
norm
  \begin{equation}
  \label{def_norm_Triebel}
  \norm{u}{H_p^a}=\norm{\FF^{-1}( a\FF u)}{L_p(\R^d)}.
  \end{equation}
These spaces were introduced and studied by Triebel in
\cite{Tr}, in a slightly more general
setting involving another parameter
$q$ (under a different form \cite[Def. 2.3/4]{Tr}, but Theorem 5.1/2 and Remark 5.1 there
shows that it is equivalent to the previous description for $q=2$).

Among other things, Triebel proved the following  results
concerning these spaces:

\begin{lem}\label{14}
For any $a\in M$ and $1<p<\infty$, the space $S$ is contained in $H_p^a$, and
dense.
\end{lem}
\begin{proof}
This is proved in Theorem 3.2/2 and Remark 3.2/2 in \cite{Tr}.
\end{proof}

For $t$, $t_-\in \R$, the function
$a_{t,t_-}(\xi,\eta)=(1+|\xi|^2+|\eta|^2)^{t/2}
(1+|\eta|^2)^{t_-/2}$ belongs to $M$. Then $H_p^{t,t_-}$ from
Definition~ \ref{def:space} is just $H_p^{a_{t,t_-}}$, and the
previous lemma says that $S$ is dense in $H_p^{t,t_-}$.

\begin{prop}[Interpolation]
\label{Triebinterpol} For any $a_0$, $a_1 \in M$, $p_0$, $p_1\in
(1,\infty)$ and $\theta \in (0,1)$, the interpolation space
$[H_{p_0}^{a_0}, H_{p_1}^{a_1}]_\theta$ is equal to $H_p^a$ for
$a=a_0^{1-\theta} a_1^\theta$ and
$1/p=(1-\theta)/p_0+\theta/p_1$.
\end{prop}
\begin{proof}
This is \cite[Theorem 4.2/2]{Tr}.
\end{proof}

We will also use the following straightforward lemma.
(Note that if $a \in M$ then $1/a \in M$, see e.g.
\cite[Lemma 2.1/1]{Tr}).
\begin{lem}[Duality]
\label{lem:duality} For any $a\in M$ and $1<p<\infty$, the dual
of the space $H_p^a$ is $H_{p'}^{1/a}$ for $1/p+1/{p'}=1$.
\end{lem}

\subsection{Multiplier theorems}

In order to understand the spaces $H_p^a$, an essential tool is
provided by  Fourier multiplier theorems. The
following Marcinkiewicz multiplier theorem (see
e.g.~\cite[Theorem 2.4/2]{Tr}) will be sufficient for our
purposes.
\begin{thm}
\label{thm:Marc} Let $b\in C^d(\R^d)$ satisfy $|\zeta^\gamma
D^\gamma b(\zeta)| \leq B$ for all multi-indices
$\gamma=(\gamma_1,\dots,\gamma_d)$ with $\gamma_j\in \{0,1\}$,
and all $\zeta\in \R^d$. Then, for all $p\in (1,\infty)$, there
exists a constant $C(p,d)$ such that, for any $u\in L_p$,
  \begin{equation}
  \norm{ \FF^{-1}( b\FF u)}{L_p} \leq CB \norm{u}{L_p}.
  \end{equation}
\end{thm}

%%%%%%%%%%%%%%%%%%%%%%%

\section{Towards Lasota-Yorke bounds on the local space $H_p^{t,t_-}$}
\label{sec:local}

Aiming at the proof of Theorem \ref{thm:MainSpectralThm} on
transfer operators, we describe in Subsections
\ref{subsec:mult} and \ref{subsec:comp}  how the local spaces
$H_p ^{t,t_-}$, which are  the building blocks of our spaces of
distributions, behave under  multiplication by a smooth
function or by the characteristic function of a nice set, as
well as under composition with a smooth map preserving the
stable leaves. Then, in Subsection~ \ref{locall}, we state and
prove a localization principle on $H_p^{t,t_-}$ that we were
not able to find in the literature and which plays a key part
in the ``zooming" procedure in the proof of Theorem
\ref{thm:MainSpectralThm}. Note for further use that since
$X_0$ is compact, \cite[Lemma 2.2]{baladi:Cinfty} (e.g.) gives
that the inclusion $\HH_p^{t,t_-}\subset \HH_p^{t',t'_-}$ for
$t' \leq t$ and $t'_- \leq t_-$ is compact if $t' < t$.

To study $H_p^{t,t_-}$, we will mainly study $H_p^{t,0}$ and
$H_p^{0,t_-}$ and use interpolation (via Proposition
\ref{Triebinterpol}). It is therefore useful to recall some
classical properties of these spaces.

When $t\geq 0$, the space $H_p^{t}$ is the classical Sobolev
space. By \cite[Theorem I.4.1]{Str}, it satisfies a Fubini
property: if $u$ is a function on $\R^d$, define a function
$u_j$ on $\R^{d-1}$ as follows:
$u_j(x_1,\dots,x_{j-1},x_{j+1},\dots,x_d)$ is the
$H_p^t(\R)$-norm of the restriction of $u$ to the line
$\{(x_1,\dots,x_{j-1},x,x_{j+1},\dots,x_d) \st x\in \R\}$. Then
$u$ belongs to $H_p^{t}(\R^d)$ if and only if each $u_j$
belongs to $L_p(\R^{d-1})$, and the norms $\norm{u}{H_p^{t}}$
and $\sum_{j=1}^d \norm{u_j}{L_p}$ are equivalent. (This is
true for any set of coordinates, but for simplicity we shall
use a fixed system of coordinates.) This makes it often
possible to study only the one-dimensional situation, and
extend it readily to $d$ dimensions.

For $t_->0$, the space $H_p^{0,t_-}$ also has a Fubini-type
property: the norm $\norm{u}{H_p^{0,t_-}}$ is equivalent to
$\sum_{j=d_u+1}^d \norm{u_j}{L_p}$ where $u_j$ is the
$H_p^{t_-}(\real)$-norm of a restriction of $u$ as above (the
proof of \cite[Theorem I.4.1]{Str} directly applies,  we may
take any coordinates on $\real^d$ which preserve the stable
leaves of the original coordinate system used to define
$H_p^{0,t_-}$, for simplicity we shall fix this original
coordinate system). In particular, the study of $H_p^{0,t_-}$
reduces to the study of the usual Sobolev space in one
dimension.

Finally, for $t_-\in \R$, the space $H_p^{0,t_-}$ also has a
slightly different Fubini-type property. Let $u$ be a function
on $\R^d$, and define a function $v$ on $\R^{d_u}$ as follows:
$v(x)$ is the $H_p^{t_-}(\R^{d_s})$-norm of the restriction of
$u$ to $\{x\}\times \R^{d_s}$. Then
$\norm{u}{H_p^{0,t_-}(\R^d)}=\norm{v}{L_p(\R^{d_u})}$: this
follows from the fact that the function $(1+|\eta|^2)^{t_-/2}$
does not depend on the variable $\xi$, which makes it possible
to integrate away the variable $x$ using the Fourier inversion
formula (see \cite[p.~1045]{Str} for details).

We will refer to these properties respectively as the
one-dimensional and the $d_s$-dimensional Fubini properties of
$H_p^{0,t_-}$.

\subsection{Multiplication by functions}
\label{subsec:mult}

\begin{lem}
\label{Leib} Let $t>0$, $t_-<0$ and $\alpha>0$ be real numbers
with $t+|t_-|<\alpha$. For any $p \in (1,\infty)$, there
exists a constant $\Cs$ such that for any $C^{\alpha}$ function
$g : \R^d \to \C$, for any distribution $u\in H_p^{t,t_-}$, the
distribution $gu$ also belongs to $H_p^{t,t_-}$ and satisfies
  \begin{equation*}
  \norm{ g \cdot u}{H_p^{t,t_-}}\le \Cs
  \|g\|_{C^{\alpha}} \norm{u}{H_p^{t,t_-}}.
  \end{equation*}
\end{lem}
The assertion $g u\in H_p^{t,t_-}$ should be interpreted as
explained after Theorem~\ref{thm:MainSpectralThm}.

\begin{proof}
Let $t^0=t+|t_-|$, $t^0_-=-t^0$ and $\theta=t/t^0$, so that
$(t,t_-)=(\theta t^0, (1-\theta) t^0_-)$ and $\max(t^0,
|t^0_-|)<\alpha$. We will write $H_p^{t,t_-}$ as an
interpolation space with parameter $\theta$ between $H_p^{t^0}$
and $H_p^{0,t^0_-}$, thereby reducing the proof to the study of
$H_p^{t^0}$ and $H_p^{0,t^0_-}$.

First, since $H_p^{t^0}$ is the classical Sobolev space,
\cite[Corollary 4.2.2]{Trie} shows that
  \begin{equation}
  \label{eq:BorneHpt}
  \norm{gu}{H_p^{t^0}} \leq \Cs \norm{g}{C^\alpha} \norm{u}{H_p^{t^0}},
  \end{equation}
where $\Cs$ depends only on $t^0$ and $\alpha$, whenever
$|t^0|<\alpha$.

Together with the $d_s$-dimensional Fubini-type property of
$H_p^{0,t^0_-}$, this readily implies
  \begin{equation}
  \label{eq:Mult0t}
  \norm{gu}{H_p^{0,t^0_-}} \leq \Cs\norm{g}{C^\alpha} \norm{u}{H_p^{0,t^0_-}}
  \end{equation}
whenever $|t^0_-|<\alpha$.

Interpolating between \eqref{eq:BorneHpt} and \eqref{eq:Mult0t}
via Proposition~\ref{Triebinterpol}, we get the conclusion of
the lemma.
\end{proof}

The following extension of a classical result of Strichartz is the key to our
results:

\begin{lem}
\label{lem:multiplier} Let $1<p<\infty$ and $1/p-1<t_-\leq 0
\leq t <1/p$. There exists a constant $\Cs$ satisfying the
following property. Let $O$ be a set in $\R^d$ whose
intersection with almost every line parallel to a coordinate
axis has at most $N$ connected components. Then, for any  $u\in
H_p^{t,t_-}$, the distribution $1_O u$ also belongs to
$H_p^{t,t_-}$, and satisfies
  \begin{equation}
  \norm{1_O u}{H_p^{t,t_-}} \leq \Cs N
  \norm{u}{H_p^{t,t_-}}.
  \end{equation}
\end{lem}

\begin{proof}
If $t_-=0$ and $t \in [0, 1/p)$ then our claim is just
Strichartz' result \cite[Cor II.4.2]{Str} on generalized
Sobolev spaces (noting that \cite[Cor II.3.7]{Str} gives the
estimate $\Cs N$). (See also \cite[\S 4.6.3]{RS} for
alternative sufficient conditions on $O$ and $p$, $t$ ensuring
that $1_O$ is a multiplier of $H_p^{t,0}$.)

Assume now that $t=0$ and $t_- \in (0, 1/p) $. Then the
one-dimensional Fubini-type argument of Strichartz \cite[Thm
I.4.1]{Str} applies, and allows us to generalize \cite[Cor
II.4.2]{Str} to give the claim. If $t=0$ and $t_- \in (1/p-1,
0)$, the result follows by duality.

Interpolating via Proposition \ref{Triebinterpol}, the set of parameters
$(1/p,t,t_-)$ for which the conclusion of the lemma holds is
convex. It therefore contains the convex hull of $\{(1/p,t,0)
\st 0\leq t<1/p\}$ and $\{ (1/p,0,t_-)\st 1/p-1 < t_- \leq
0\}$, which coincides with the set $\{ (1/p,t,t_-) \st
1/p-1<t_-\leq 0 \leq t <1/p\}$.
% Indeed, if $1/p-1<t_-\leq 0\leq t<1/p$, then $(1/p,t,t_-)$ is located
% on the segment between $(1/p+|t_-|, t+|t_-|, 0)$ and  $(1/p-t, 0, t_--t)$,
% which belong to the aforementioned sets.
\end{proof}

\subsection{Composition with smooth maps preserving the stable leaves}
\label{subsec:comp}

In this paragraph, we study the behavior of $H_p^{t,t_-}$ under
the composition with smooth maps preserving the stable leaves.

Let us start with a very rough and easy to prove lemma.

\begin{lem}
\label{lem:CompositionFacile} Let $1<p<\infty$,  and $t$, $t_-$
be real numbers with $|t|+|t_-|\le 1$. There exists a constant
$\Cs$ such that, for any invertible matrix $A$ on $\R^d$,
sending $\{0\}\times \R^{d_s}$ to itself, and for any $u\in
H_p^{t,t_-}$,
  \begin{equation}
  \norm{ u\circ A}{H_p^{t,t_-}} \leq \Cs |\det A|^{-1/p}
  \max(\nor{A},\nor{A^{-1}}) \norm{u}{H_p^{t,t_-}}.
  \end{equation}
\end{lem}
\begin{proof}
By \cite[Proposition 2.1.2 (iv)+(vii)]{RS}, the $H_p^1$-norm is
equivalent to the norm $\norm{u}{L_p}+\norm{Du}{L_p}$. Hence,
$\norm{ u\circ A}{H_p^{1,0}} \leq \Cs |\det A|^{-1/p}
\max(\nor{A},\nor{A^{-1}}) \norm{u}{\HH_p^{1,0}}$. Similarly,
$\norm{ |\det A |^{-1} u\circ A^{-1}}{H_{p'}^{0,1}} \leq \Cs
|\det A|^{-1+1/{p'}} \max(\nor{A},\nor{A^{-1}})
\norm{u}{H_{p'}^{0,1}}$, by a $d_s$-dimensional Fubini-type
argument. Since the adjoint of $u \mapsto \det A ^{-1} u\circ
A^{-1}$ is $u \mapsto u \circ A$, the general case follows by
duality (Lemma~ \ref{lem:duality}) and interpolation
(Proposition~\ref{Triebinterpol}).
\end{proof}

\begin{lem}
\label{lem:CompositionDure} Let $\alpha\in (0,1)$, let $F:\R^d
\to \R^d$ be a $C^{1+\alpha}$ diffeomorphism sending stables
leaves to stable leaves, and let $A$ be a matrix such that, for
all $z\in \R^d$, $\nor{A^{-1}\circ DF(z)}\leq 2$ and
$\nor{DF(z)^{-1}\circ A}\leq 2$.

Assume moreover that $A$ can be written as
$M_0^{-1}\left(\begin{array}{cc} A^u & 0 \\ 0 & A^s
\end{array}\right)M_1$, where $M_0$ and $M_1$ are matrices sending stable leaves to
stable leaves, and $\mu_u:=\nor{A^u}\leq 1$,
$\mu_s:=\nor{(A^s)^{-1}}^{-1}\geq 1$.\footnote{The matrix norms
are the operator norms with respect to the usual euclidean
metric on $\R^d$, so that the norm of a matrix equals the norm
of its transpose.}

Then, for all $t>0$ and $t_-<0$ with $t+|t_-|<\alpha$ and
$t+t_-<0$, for all $p\in (1,\infty)$, there exists a constant
$\Cs$ depending only on
$\max(\nor{M_0},\nor{M_0^{-1}},\nor{M_1},\nor{M_1^{-1}})$ and
$t$, $t_-$, $p$, and a constant $C(A,F)$ such that, for all
$u\in H_p^{t,t_-}$,
  \begin{multline*}
  \norm{u\circ F}{H_p^{t,t_-}}\leq \Cs
  \norm{\det A / \det DF}{C^\alpha} |\det A|^{-1/p}
  \max(\mu_u^t, \mu_s^{t+t_-})\norm{u}{H_p^{t,t_-}}
  \\ + C \norm{u}{H_p^{0,t_-}}.
  \end{multline*}
\end{lem}
In the applications to transfer operators, $F$
will be the local {\it inverse} of some iterate $T^n$ of a piecewise
hyperbolic map. Since $T^n$ is contracting along $E^s$ and
expanding along $E^u$, the map $F$ will
therefore satisfy the assumptions of the lemma regarding
$\mu_s$ and $\mu_u$.

\begin{proof}[Proof of Lemma \ref{lem:CompositionDure}]
We will write $u\circ F=u\circ A\circ A^{-1}\circ F$. Hence, we
need to study the composition with $A$ and $A^{-1}\circ F$. We
claim that
  \begin{equation}
  \label{ComposeA}
  \norm{u\circ A}{H_p^{t,t_-}} \leq |\det A|^{-1/p}
  \Cs \max(\mu_u^t, \mu_s^{t+t_-})\norm{u}{H_p^{t,t_-}}+ C
  \norm{u}{H_p^{0,t_-}}
  \end{equation}
and
  \begin{equation}
  \label{ComposeTA}
  \norm{u\circ A^{-1}\circ F}{H_p^{t,t_-}} \leq \Cs
  \norm{\det A / \det DF}{C^\alpha}
  \norm{u}{H_p^{t,t_-}}.
  \end{equation}
Together, these equations prove the lemma.

\emph{First step.} Let us prove \eqref{ComposeA}. This is
a special case of \cite[Lemma 2.10]{baladi:Cinfty}
(replacing $(0,t_-)$ by $(t-1/2, t_-)$). We will
give the proof for the convenience of the reader,
since it is at the same time very simple and at the heart of
our argument. Lemma \ref{lem:CompositionFacile} deals with the
composition with $M_0^{-1}$ and $M_1$, hence we can assume that
$M_0=M_1=\Id$.

We want to estimate $\norm{u\circ
A}{H_p^{t,t_-}}=\norm{\FF^{-1}(a_{t,t_-} \FF (u\circ
A))}{L_p}$. A change of variables readily gives
$\FF^{-1}(a_{t,t_-} \FF (u\circ A))=\FF^{-1}(a_{t,t_-}\circ
\transposee{A}\cdot  \FF u) \circ A$. Hence, we have to show
that
  \begin{equation}
  \label{eq:klmjwxvop}
  \norm{\FF^{-1}(a_{t,t_-}\circ
  \transposee{A} \cdot \FF u)}{L_p} \leq
  \Cs \max(\mu_u^t, \mu_s^{t+t_-})\norm{u}{H_p^{t,t_-}}+ C
  \norm{u}{H_p^{0,t_-}}.
  \end{equation}
Write $\transposee{A}=\left(\begin{array}{cc} U&0\\0 & S
\end{array}\right)$ with $|U \xi|\leq
\mu_u|\xi|$ and $|S \eta|\geq \mu_s |\eta|$ by definition of
$\mu_u,\mu_s$. Let
  \begin{equation}
  b(\xi,\eta)=a_{t,t_-}\circ
  \transposee{A}(\xi,\eta)=(1+|U \xi|^2+|S\eta|^2)^{t/2}
  (1+|S\eta|^2)^{t_-/2}.
  \end{equation}
Let us prove that, if $C$ is large enough, we have
  \begin{equation}
  \label{eq:opuispoif}
  b \leq \Cs \max(\mu_u^t, \mu_s^{t+t_-})
  a_{t,t_-}+C a_{0,t_-}.
  \end{equation}
If we can prove this equation together with the corresponding
estimates for the successive derivatives of $b$, then Theorem
\ref{thm:Marc}  applied to
$$b/(\Cs \max(\mu_u^t, \mu_s^{t+t_-})
  a_{t,t_-}+C a_{0,t_-})$$
gives
  \begin{equation}
  \norm{\FF^{-1}(b \FF u)}{L_p} \leq
  \Cs \norm{ \FF^{-1}((\Cs \max(\mu_u^{t}, \mu_s^{t+t_-})
  a_{t,t_-}+C a_{0,t_-}) \FF u)}{L_p},
  \end{equation}
which yields \eqref{eq:klmjwxvop}.

Let us now prove \eqref{eq:opuispoif} (the proof for the
derivatives of $b$ is similar). We will freely use the
following trivial inequalities: for $x\geq 1$ and $\lambda\geq
1$,
  \begin{equation}
  \frac{1}{\lambda}(1+\lambda x)\leq 1+x \leq
  \frac{2}{\lambda}(1+\lambda x).
  \end{equation}
Assume first $|U\xi|^2 \leq |S\eta|^2$ and $|S\eta|^2 \geq 1$.
Then, since $t>0$ and $t+t_-<0$,
  \begin{align*}
  b(\xi,\eta) &\leq (1+2 |S\eta|^2)^{t/2} (1+|S\eta|^2)^{t_-/2}
  \leq 2^{t/2} (1+|S\eta|^2)^{t/2}(1+|S\eta|^2)^{t_-/2}
  \\&\leq 2^{t/2} (1+\mu_s^2 |\eta|^2)^{(t+t_-)/2}
  \leq 2^{t/2} (\mu_s^2/2)^{(t+t_-)/2} (1+|\eta|^2)^{(t+t_-)/2}
  \\& \leq 2^{-t_-/2} \mu_s^{(t+t_-)} a_{t,t_-}(\xi,\eta).
  \end{align*}
If $|U\xi|^2 \geq |S\eta|^2$ and $|U\xi|^2 \geq 1$, then
  \begin{align*}
  b(\xi,\eta)& \leq (1+2 |U\xi|^2)^{t/2} (1+|S\eta|^2)^{t_-/2}
  \leq 2^{t/2} (1+|U\xi|^2)^{t/2} (1+\mu_s^2 |\eta|^2)^{t_-/2}
  \\&\leq 2^{t/2} (1+ \mu_u^2 |\xi|^2)^{t/2} (1+|\eta|^2)^{t_-/2}
  \leq 2^{t/2} (2\mu_u^2)^{t/2} (1+|\xi|^2)^{t/2} (1+|\eta|^2)^{t_-/2}
  \\&\leq 2^{t} \mu_u^t a_{t,t_-}(\xi,\eta).
  \end{align*}
In the remaining case, $\xi$ and $\eta$ are uniformly bounded,
and \eqref{eq:opuispoif} follows by choosing $C$ large enough.
This concludes the proof of \eqref{ComposeA}.

\emph{Second step.} Let us now prove \eqref{ComposeTA}. We will
write $\tilde F=A^{-1}\circ F$. As in the proof of Lemmas~
\ref{Leib}, \ref{lem:multiplier}, and ~ \ref{lem:CompositionFacile},
we will  study simpler spaces before
concluding by interpolation. We thus write $(t,t_-)=(\theta
t^0, (1-\theta) t^0_-)$ for some $0<\theta<1$ and $t^0, -t^0_-
\in (0,\alpha)$.

By \cite[Proposition 2.1.2 (iv)+(vii)]{RS},
the $H_p^1$-norm is
equivalent to the norm $\norm{u}{L_p}+\norm{D u}{L_p}$.
Since the derivative of $\tilde F$ has norm everywhere
bounded by $2$ and $|\det D\tilde F|\le 2^d$ by assumption, we get after a change of
variables $\norm{u\circ \tilde F}{H_p^1} \leq \Cs
\norm{u}{H_p^1}$. Since $\norm{u\circ \tilde F}{L_p} \leq \Cs
\norm{u}{L_p}$, the interpolation inequality \eqref{interpp}
gives
  \begin{equation}
  \label{eq:Interp0}
  \norm{u\circ \tilde F}{H_p^{t^0}} \leq \Cs \norm{u}{H_p^{t^0}}.
  \end{equation}

Applying the same argument via Fubini to $\tilde F^{-1}$ on each leaf of the vertical
direction, we also have
$\norm{ u \circ \tilde F^{-1}}{H_{p'}^{0,1}}
\leq \Cs \norm{u}{H_{p'}^{0,1}}$.  The adjoint of the composition
by $\tilde F^{-1}$ is given by $\Pp(u)= \det D\tilde F
\cdot u\circ \tilde F$. Hence, duality yields $\norm{\Pp
u}{H_p^{0,-1}} \leq \Cs \norm{u}{H_p^{0,-1}}$. Since $\Pp$ is
bounded by $\Cs$ on $L_p$, we  get by interpolation
  \begin{equation}
  \label{qoisufmkljqsdf}
  \norm{\det D\tilde F \cdot u\circ \tilde F}{H_p^{0,t^0_-}}\leq \Cs \norm{u}{H_p^{0,t^0_-}}.
  \end{equation}
Together with \eqref{eq:Mult0t}, we obtain
  \begin{equation}
  \label{eq:Interp1}
  \begin{split}
  \norm{u\circ \tilde F}{H_p^{0,t^0_-}}&\leq \Cs\norm{1/\det D\tilde F}{C^\alpha}
  \norm{\det D\tilde F \cdot u\circ \tilde F}{H_p^{0,t^0_-}}
  \\&
  \leq \Cs\norm{1/\det D\tilde F}{C^\alpha} \norm{u}{H_p^{0,t^0_-}}.
  \end{split}
  \end{equation}

Interpolating between \eqref{eq:Interp0} and
\eqref{eq:Interp1}, we get
  \begin{equation}
  \norm{u\circ \tilde F}{H_p^{t,t-}}\leq
  \Cs \norm{1/\det D\tilde F}{C^\alpha}^{1-\theta} \norm{u}{H_p^{t,t-}}.
  \end{equation}
Finally, $1/\det D\tilde F=\det A / \det DF$ is bounded from
below, and \eqref{ComposeTA} follows.
\end{proof}

\begin{rmk}[Invariance]
\label{rmk:Invariance} The arguments in the second step of the
proof of Lemma \ref{lem:CompositionDure} (with $A=\Id$) also
imply that, whenever $t>0$ and $t_-<0$ satisfy
$t+|t_-|<\alpha$, then the space $H_p^{t,t_-}$ is invariant
under the composition with $C^{1+\alpha}$ diffeomorphisms of
$\R^d$ sending stable leaves to stable leaves.
\end{rmk}

\begin{rmk}[Extending  \cite{baladi:Cinfty}
to $C^{1+\alpha}$ Anosov diffeomorphisms]
\label{extend}
If $0<\alpha <1$ we can apply  Lemma~
\ref{lem:CompositionDure} .
If $\alpha \ge 1$ and $t >0$, $t+t_- <0$
satisfy $t+|t_-|<\alpha$,  letting $m$ be the smallest integer
$\ge t+|t_-|$, \cite[Proposition 2.1.2 (iv)+(vii)]{RS},
implies that
the $H_p^m$-norm is
equivalent to the norm $\sum_{|\gamma|\le m}\norm{\partial^\gamma u}{L_p}$.
Thus, replacing the matrix $A$ in Lemma~
\ref{lem:CompositionDure} by a $C^\infty$ diffeomorphism $A$ preserving stable
leaves, with least expansion $\mu_s\ge 1$
on the verticals, and whose inverse preserves horizontal cones with
least expansion $\mu_s^{-1}\ge 1$,
and such that $\norm{DA^{-1}\circ DF}{C^{m-1}}\le 2$ and $\norm{DF^{-1} \circ DA}{C^{m-1}}\le 2$,
we get,
by applying \cite[Lemma 2.10]{baladi:Cinfty} to
prove the analogue of  (\ref{ComposeA}),
that
\begin{multline*}
  \norm{u\circ F}{H_p^{t,t_-}}\leq \Cs
  \norm{\det DA / \det DF}{C^\alpha} |\det DA|^{-1/p}
  \max(\mu_u^t, \mu_s^{t+t_-})\norm{u}{H_p^{t,t_-}}
  \\ + C \norm{u}{H_p^{t-1/2,t_-}}.
  \end{multline*}
The proof of Theorem~ \ref{thm:MainSpectralThm}
then applies to any  $C^{1+\alpha}$ Anosov diffeomorphism $T$
with $C^{1+\alpha}$ stable distribution, and to any $C^\alpha$ weight $g$,
with $\alpha >0$.
\end{rmk}

\subsection{Localization}\label{locall}

\begin{lem}[Localization principle]
\label{lem:localization}
Let $\eta:\R^d\to [0,1]$ be a $C^\infty$ function with compact
support and write $\eta_m(x)=\eta(x+m)$. For any $p\in
(1,\infty)$ and $t$, $t_-\in \R$,  there exists $\Cs >0$ so
that for each  $u \in H_p^{t,t_-}$
  \begin{equation}
  \left(\sum_{m\in \Z^d} \norm{\eta_m u }{H_p^{t,t_-}}^p\right)^{1/p}
  \le \Cs  \norm{u}{H_p^{t,t_-}}.
  \end{equation}
\end{lem}

\begin{rmk}
If, in addition to the assumptions of
Lemma~\ref{lem:localization}, one supposes that $\sum_{m \in
\integer^d} \eta_m(x) =1$ for all $x$, then one can show that
there is $\Cs$ so that for each $u$ such that ${ \eta_m u }\in
{H_p^{t,t_-}}$ for all $m$ we have
$$ \norm{u}{H_p^{t,t_-}}\le
\Cs \left(\sum_{m\in \Z^d} \norm{\eta_m u }{H_p^{t,t_-}}^p\right)^{1/p}.
$$
(We shall not need the above bound.)
\end{rmk}

\begin{proof}[Proof of Lemma \ref{lem:localization}]
For $t_-=0$ and arbitrary $t$, Lemma~\ref{lem:localization} is
a result of Triebel \cite[Theorem 2.4.7]{Trie} based on a
Paley-Littlewood-type decomposition. Moreover, the constant
$\Cs$ depends only on the size of the support of $\eta$, and
its $C^k$-norm for some large enough $k$.

To handle $t_-\in \real$, we will (again) start from the result
for the classical Sobolev space and use Fubini and
interpolation, as follows.

Let us prove the lemma for $t=0$ and $t_-\in \R$, using a
$d_s$-dimensional Fubini argument. We have
  \begin{equation}
  \sum_{m\in \Z^d} \norm{\eta_m u}{H_p^{0,t_-}(\R^d)}^p
  = \sum_{m\in \Z^d} \int_{x\in \R^{d_u}} \norm{ \eta_m u}{H_p^{t_-}(\{x\}\times \R^{d_s})}^p \dd x.
  \end{equation}
For each $x\in \R^{d_u}$, the values of $m\in \Z^d$ for which
the restriction of $\eta_m u$ to $\{x\}\times \R^{d_s}$ is
nonzero are contained in a set $M(x) \times \Z^{d_s}$, where
$\Card M(x)$ is bounded independently of $x$. Using the result
of Triebel for the Sobolev space $H_p^{t_-}(\R^{d_s})$, we get
  \begin{equation}
  \sum_{m\in \Z^d}\norm{ \eta_m u}{H_p^{t_-}(\{x\}\times \R^{d_s})}^p
  \leq \Cs \norm{u}{H_p^{t_-}(\{x\}\times \R^{d_s})}^p.
  \end{equation}
Integrating over $x\in \R^{d_u}$ and using the Fubini equality
  \begin{equation}
  \int_{x\in \R^{d_u}} \norm{u}{H_p^{t_-}(\{x\}\times \R^{d_s})}^p \dd x=\norm{u}{H_p^{0,t_-}}^p,
  \end{equation}
we obtain the lemma for $t=0$ and $t_-\in \R$.

Consider the map $u\mapsto (\eta_m u)_{m\in \Z^d}$. We have
shown that it sends continuously $H_p^t$ to $\ell_p(H_p^t)$ and
$H_p^{0,t_-}$ to $\ell_p(H_p^{0,t_-})$. By interpolation, for
any $\theta\in (0,1)$, it sends $[H_p^t, H_p^{0,t_-}]_\theta$
to $[\ell_p(H_p^t), \ell_p(H_p^{0,t_-})]_\theta$. By
Proposition \ref{Triebinterpol}, the first space is
$H_p^{(1-\theta)t, \theta t_-}$ while, by \cite[Theorem
1.18.1]{TrB} and again Proposition \ref{Triebinterpol}, the
second space is $\ell_p(H_p^{(1-\theta)t, \theta t_-})$. This
proves the lemma.
\end{proof}

%%%%%%%%%%%%%%%%%%%%%%%%%%%%%%%%%%

%%%%%%%%%%%%%%%%%%%%%%%%%%%%%%%%%%%%%%%%%%%%

\section{Proof of the main theorem}
\label{mainsec}

In this section, we prove Theorem \ref{thm:MainSpectralThm}.
Let us fix once and for all a piecewise $C^{1+\alpha}$
hyperbolic map $T$ and  a $C^\alpha$ function $g$, satisfying the
assumptions of this theorem. We will denote by $\Cs$ constants
that depend only on $p$, $t$, $t_-$ and $T$.

We recall that the norm on $\HH_p^{t,t_-}$ has been defined in
\eqref{defnorm} using a partition of unity
$\rho_1,\dots,\rho_J$ and charts $\kappa_1,\dots,\kappa_J$
subordinated to this partition of unity.

In the following arguments, when working on a set
$\overline{O_\ii}$ or in a neighborhood of this set (with $\ii$
of length $n$), then $T^n$ will implicitly mean $T_\ii$. In the
same way, $g^{(n)}$ will rather be a smooth extension of
$g^{(n)}\big|_{O_\ii}$ to a neighborhood of $\overline{O_\ii}$.
This should not cause any confusion.

To study $\Lp_g^n$, we will need, in addition to the
estimates from Section~ \ref{sec:local}, to iterate the inverse
branches $T_\ii^{-1}$, to truncate the functions and to use
partitions of unity. To do this, we will use the three
following lemmas.

\begin{lem}
\label{lem:iterate} There exists a constant $\Cs$ such that,
for any $n$ and $\ii=(i_0,\dots,i_{n-1})$, for any $x\in
\overline{O_\ii}$, for any $j,k\in [1,J]$ such that $x\in \supp
\rho_j$ and $y=T_\ii x\in\supp \rho_k$, there exists a
neighborhood $O$ of $y$ and a $C^{1+\alpha}$ diffeomorphism $F$
of $\R^d$, coinciding with $\kappa_j \circ T_\ii^{-1}\circ
\kappa_k^{-1}$ on $\kappa_k(O)$, and satisfying the assumptions
of Lemma \ref{lem:CompositionDure} with $\mu_u \leq \Cs
\lambda_{u,n}^{-1}(x)$ and $\mu_s \geq
\Cs^{-1}\lambda_{s,n}^{-1}(x)$, and
$$\max(\nor{M_0},\nor{M_0^{-1}},\nor{M_1},\nor{M_1^{-1}})\leq
\Cs.
$$
\end{lem}
\begin{proof}
Let $F_0=\kappa_j \circ T_\ii^{-1}\circ \kappa_k^{-1}$, it is
defined on a neighborhood of $\kappa_k(y)$. Moreover, let $P$
be a $d_u$-dimensional subspace of the unstable cone at $x$,
and let $M_0$, $M_1$ be invertible matrices (with bounded
norms) sending respectively $D\kappa_j(x)P$ and $D\kappa_k(y)
DT_\ii(x)P$ to $\R^{d_u}\times \{0\}$, and stable leaves to
stable leaves. Such matrices exist since the unstable cone is
uniformly bounded away from the stable direction.

Let $A=DF_0(\kappa_k (y))$, then $M_0 AM_1^{-1}$ sends
$\R^{d_u}\times \{0\}$ to itself, and $\{0\}\times \R^{d_s}$ to
itself, i.e., it is block-diagonal. Hence, the matrix $A$
satisfies the assumptions of Lemma \ref{lem:CompositionDure}.
Let $F$ be a $C^{1+\alpha}$ diffeomorphism of $\R^d$ coinciding
with $F_0$ on a neighborhood of $\kappa_k(y)$ and such that
$DF(z)$ is everywhere close to $A$. Up to taking a smaller
neighborhood $O$ of $y$ (depending on $n$), the claims of
Lemma~ \ref{lem:iterate} hold for $F$.
\end{proof}

\begin{lem}
\label{lem:truncate} There exists  $\Cs$ such that,
for any $n$, for any $\ii=(i_0,\dots,i_{n-1})$, for any $x\in
\overline{O_\ii}$, for any $j$ such that $x\in \supp \rho_j$,
there exists a neighborhood $O'$ of $x$ and a matrix $M$
sending stable leaves to stable leaves, with
$$\max(\nor{M},\nor{M^{-1}})\leq \Cs,$$
such that the intersection of $M \kappa_j(O'\cap O_\ii)$
intersects almost any line parallel to a coordinate axis along
at most $\Cs n$ connected components.
\end{lem}
\begin{proof}
Let $L$ be as in Definition \ref{def:WTC}. Fix
$\ii=(i_0,\dots,i_{n-1})$ and $x\in \overline{O_\ii}$. Let
$a_1,\dots,a_d$ be a basis of $\T_x X$, which is close to an
orthonormal basis, such that its last $d_s$ vectors form a
basis of $E^s(x)$. We can ensure that, for any $\ell<n$,
$DT^\ell(x)a_k$ is $L$-generic with respect to $\partial
O_{i_j}$, for $d_u<k\leq d$. This is indeed a consequence of
the definition of weak transversality. Moving slightly the
vectors $a_k$ for $1\leq k\leq d_u$, we can also ensure that
$DT^\ell(x)a_k$ is transversal to the hypersurfaces defining
$\partial O_{i_j}$ at $T^\ell x$ for any $\ell < n$.

Let $b_k =D\kappa_j(x) \cdot a_k$, so that $b_1,\dots,b_d$ is a
basis of $\R^d$. Multiplying $a_k$ by a scalar, we can ensure
that $b_k$ has norm $1$. If $O'$ is a small enough neighborhood
of $x$, then $\kappa_\ell(O'\cap O_\ii)$ intersects almost any
line oriented by one of the vectors $b_k$, $d_u<k\leq d$, along
at most $nL$ connected components, by definition of
$L$-genericity. Moreover, it intersects any line oriented by
one of the vectors $b_k$, $1\leq k\leq d_u$, along at most one
connected component by construction.

Let $M$ be the matrix sending $b_1,\dots,b_d$ to the canonical
basis of $\R^d$, it satisfies the requirements of the lemma.
\end{proof}

The following lemma on partitions of unity is similar to \cite[Lemma 7.1]{BT1}.
\begin{lem}
\label{lem:sum} Let $t$ and $t_-$ be
arbitrary real numbers. There exists a constant $\Cs$ such that, for
any distributions $v_1,\dots, v_l$ with compact support in
$\R^d$, belonging to $H_p^{t,t_-}$, there exists a constant $C$
with
  \begin{equation}
  \norm{ \sum_{i=1}^l v_i}{H_p^{t,t_-}}^p \leq \Cs m^{p-1} \sum_{i=1}^l
  \norm{v_i}{H_p^{t,t_-}}^p + C \sum_{i=1}^l
  \norm{v_i}{H_p^{t-1,t_-}}^p,
  \end{equation}
where $m$ is the intersection multiplicity of the supports of
the $v_i$'s, i.e., $m=\sup_{x\in \R^d} \Card\{i \st x\in
\supp(v_i)\}$.
\end{lem}
\begin{proof}
Let $A$ be the operator acting on distributions by $A
v=\FF^{-1}((1+|\xi|^2+|\eta|^2)^{t/2}(1+|\eta|^2)^{t_-/2} \FF
v)$, so that $\norm{v}{\HH_p^{t,t_-}}=\norm{Av}{L_p}$.

\cite[Lemma 2.7]{baladi:Cinfty} shows that, for any
distribution $v$ with compact support $K$ and any neighborhood
$K'$ of this support, there exist $C>0$ and a function
$\Psi:\R^d\to [0,1]$ equal to $1$ on $K$ and vanishing on the
complement of $K'$, with
  \begin{equation}
  \norm{\Psi Av - Av}{L_p} \leq C \norm{v}{H_p^{t-1,t_-}}.
  \end{equation}

Let $v_1,\dots,v_l$ be distributions with compact supports
whose intersection multiplicity is $m$. Choose neighborhoods
$K'_1,\dots,K'_l$ of the supports of the $v_i$s
whose intersection multiplicity is also $m$, and functions
$\Psi_1,\dots,\Psi_l$ as above. Then
  \begin{equation}\label{previous}
  \norm{\sum_i v_i}{H_p^{t,t_-}}^p= \norm{\sum_i Av_i}{L_p}^p
  \leq \norm{\sum_i \Psi_i Av_i}{L_p}^p +
  C\sum_i\norm{v_i}{H_p^{t-1,t_-}}^p.
  \end{equation}
By convexity, the inequality $(x_1+\dots+x_m)^p \leq
m^{p-1}\sum x_i^p$ holds for any nonnegative numbers
$x_1,\dots,x_m$. Since the multiplicity of the $K'_i$s is at
most $m$, this yields
  \begin{equation}
  \left|\sum_i \Psi_i Av_i\right|^p \leq m^{p-1} \sum_i |Av_i|^p.
  \end{equation}
Integrating this inequality and using (\ref{previous}), we
get the lemma.
\end{proof}

\begin{proof}[Proof of Theorem \ref{thm:MainSpectralThm}]
Let $p$, $t$ and $t_-$ be as in the assumptions of the theorem.
Let $n>0$, and let $r_n>1$ (the precise value of $r_n$ will be
chosen later). We define a dilation $R_n$ on $\R^d$ by $R_n(z)=r_n z$.
Let $\norm{u}{n}$ be another norm on $\HH_p^{t,t_-}$,
given by
  \begin{equation}\label{zoom}
  \norm{u}{n}=\sum \norm{(\rho_j u)\circ \kappa_j^{-1} \circ R_n^{-1}}{H_p^{t,t_-}}.
  \end{equation}
The norm $\norm{u}{n}$ is of course equivalent to the usual norm on
$\HH_p^{t,t_-}$, but  we
look at the space $X_0$ at a smaller scale. Functions are
much more flatter at this new scale, so that estimates involving their $C^\alpha$ norm, such as
Lemma \ref{Leib} or Lemma~\ref{lem:CompositionDure}, will not cause
problems. This will also enable us to use partitions of
unity with very small supports without spoiling the estimates.
The use of this ``zooming" norm is similar to the good choice of
$\epsilon_0$ in \cite{Sau}, or the use of weighted norms in
\cite{DL}.

We will prove that, if $n$ is fixed and $r_n$ is large enough,
then
  \begin{multline}
  \label{eq:main}
  \norm{\Lp_g^n u}{n}^p\leq C \norm{u}{\HH_p^{0,t_-}}^p\\+ \Cs n^p
  D_n^b (D_n^e)^{p-1}
  \norm{ |\det DT^n| \max(\lambda_{u,n}^{-(t+t_-)},
  \lambda_{s,n}^{-t_-})^p
  |g^{(n)}|^p}{L_\infty}\norm{u}{n}^p.
  \end{multline}
The injection of $\HH_p^{t,t_-}$ into $\HH_p^{0,t_-}$ is
compact. Hence, by Hennion's theorem \cite{He}, the essential
spectral radius of $\Lp_g^n$ acting on $\HH_p^{t,t_-}$ (for
either $\norm{u}{\HH_p^{t,t_-}}$ or $\norm{u}{n}$, since these
norms are equivalent) is at most
  \begin{equation}
  \Bigl[\Cs n^p
  D_n^b (D_n^e)^{p-1}
  \norm{ |\det DT^n| \max(\lambda_{u,n}^{-(t+t_-)},
  \lambda_{s,n}^{-t_-})^p
  |g^{(n)}|^p}{L_\infty}\Bigr]^{1/p}.
  \end{equation}
Taking the power $1/n$ and letting $n$ tend to $\infty$, we
obtain Theorem \ref{thm:MainSpectralThm} since the quantity
$(\Cs n^p)^{1/pn}$ converges to $1$ (here, it is essential that
$\Cs$ does not depend on $n$).

\smallskip

It remains to prove \eqref{eq:main}, for large enough $r_n$.
The estimate will be subdivided into three steps:
\begin{enumerate}
\item Decomposing $u$ into a sum of distributions $v_{j,m}$ with small supports and
well controlled $\norm{\cdot}{n}$ norms.
\item Estimating each term $(1_{O_{\ii}} g^{(n)} v_{j,m})\circ
T_{\ii}^{-1}$, for $\ii$ of length $n$.
\item Adding all  terms to obtain $\Lp_g^n u$.
\end{enumerate}

\emph{First step.} For $1\leq j\leq J$ and $m\in \Z^d$, let
$\tilde v_{j,m}=\eta_m \cdot(\rho_j u)\circ \kappa_j^{-1}\circ
R_n^{-1}$, where $\eta_m(x)=\eta(x+m)$, with $\eta: \real^d \to [0,1]$  a
compactly supported $C^\infty$ function so that
$\sum_{m \in \Z^d} \eta_m=1$.
Since the intersection multiplicity of the supports of the functions $\eta_m$
is bounded, this is also the case for the $\tilde v_{j,m}$.
Moreover, if $j$ is fixed, we get using Lemma~\ref{lem:localization}
  \begin{equation}
  \label{eq:decompose}
  \begin{split}
  \sum_{m\in \Z^d} \norm{\tilde v_{j,m}}{H_p^{t,t_-}}^p&
  =\sum_{m\in \Z^d} \norm{\eta_m \cdot (\rho_j u)\circ \kappa_j^{-1}\circ R_n^{-1}}{H_p^{t,t_-}}^p
  \\&
  \leq \Cs \norm{(\rho_j u)\circ \kappa_j^{-1}\circ R_n^{-1}}{H_p^{t,t_-}}^p
 \leq \Cs \norm{u}{n}^p.
  \end{split}
  \end{equation}
Since $R_n$ expands the distances by a factor $r_n$ while the
size of the supports of the functions $\eta_m$ is uniformly
bounded, the supports of the distributions
$$v_{j,m}= \tilde v_{j,m}\circ R_n \circ \kappa_j=\eta_m \circ R_n \circ \kappa_j\cdot(\rho_j u)
$$
are arbitrarily small if $r_n$ is large enough.
Finally
  \begin{equation}
  u=\sum_j \rho_j u = \sum_{j,m}   v_{j,m}.
  \end{equation}

\emph{Second step.} Fix $j,k\in \{1,\dots,J\}$, $m\in \Z^d$ and
$\ii=(i_0,\dots,i_{n-1})$.
We will prove that
  \begin{multline}
  \label{eq:secondstep}
  \norm{ (\rho_k (g^{(n)}1_{O_\ii} v_{j,m})\circ T_\ii^{-1})\circ
  \kappa_k^{-1}\circ R_n^{-1}}{H_p^{t,t_-}} \leq C \norm{u}{\HH_p^{0,t_-}}\\
   + \Cs n \norm{ |\det DT^n|^{1/p} g^{(n)}\max(\lambda_{u,n}^{-t}, \lambda_{s,n}^{-(t+t_-)})}{L_\infty}\norm{\tilde v_{j,m}}{H_p^{t,t_-}}.
  \end{multline}

First, if the support of $v_{j,m}$ is small enough (which can
be ensured by taking $r_n$ large enough), there exists a
neighborhood $O$ of this support and a matrix $M$ satisfying
the conclusion of Lemma~ \ref{lem:truncate}: this follows from
Lemma~ \ref{lem:truncate} and the compactness  of $X_0$. Therefore, the
intersection of $R_n (M (\kappa_j(O\cap O_\ii)))$ with almost
any line parallel to a coordinate axis contains at most $\Cs n$
connected components. Hence, Lemma~\ref{lem:multiplier} implies
that the multiplication by $1_{O\cap O_\ii}\circ \kappa_j^{-1}
\circ M^{-1}\circ R_n^{-1}$ sends $H_p^{t,t_-}$ into itself,
with a norm bounded by $\Cs n$. Using the fact that $M$ and
$R_n$ commute, the properties of $M$, and Lemma~
\ref{lem:CompositionFacile}, we get
  \begin{equation}\label{v'}
  \norm{ 1_{O_\ii}\circ \kappa_j^{-1}\circ R^{-1}_n \cdot \tilde v_{j,m}}{H_p^{t,t_-}}
  \leq \Cs n \norm{\tilde v_{j,m}}{H_p^{t,t_-}}.
  \end{equation}
(Recall that $v_{j,m}$ is supported inside $O$.)
Next, let
$$\tilde v_{j,k,m}=((\rho_k \circ T_{\ii}) 1_{O_\ii})\circ
\kappa_j^{-1}\circ R^{-1}_n \cdot \tilde v_{j,m}$$ (we suppress
$\ii$ from the notation for simplicity). Let  also $\chi$ be a
$C^\infty$ function supported in the neighborhood $O$ of the
support of $v_{j,m}$ with $\chi\equiv 1$ on this support. Up to
taking larger $r_n$ we may ensure that  $\norm{(\chi (\rho_k
\circ T_\ii ))\circ \kappa_j^{-1}\circ R_n^{-1}}{C^\alpha}\le
\Cs$. Then Lemma~\ref{Leib} and \eqref{v'} imply
\begin{equation}\label{v''}
\norm{\tilde v_{j,k,m}}{H_p^{t,t_-}}
 \leq \Cs n \norm{\tilde v_{j,m}}{H_p^{t,t_-}}
 \end{equation}
 In addition, we have
  \begin{align}\label{glu}
  ((\rho_k \circ T_{\ii})1_{O_\ii} v_{j,m})\circ T_\ii^{-1}\circ \kappa_k^{-1}\circ R_n^{-1}
  &=\tilde v_{j,k,m} \circ R_n\circ \kappa_j \circ T_\ii^{-1}\circ \kappa_k^{-1}\circ R_n^{-1}\\
 \nonumber &= \tilde v_{j,k,m} \circ R_n \circ F \circ R_n^{-1},
  \end{align}
where $F$ is given by Lemma~  \ref{lem:iterate} (we use the
fact that the support of $v_{j,m}\circ T_\ii^{-1}$ is contained
in a very small neighborhood $O'$  if $r_n$ is large enough,
and again the compactness of $X_0$). The diffeomorphism $F$ satisfies the
assumptions of Lemma~ \ref{lem:CompositionDure}. Since the
dilations $R_n$ commute with any matrix, this is also the case
of the diffeomorphism $G=R_n\circ F\circ R_n^{-1}$. Applying
Lemma \ref{lem:CompositionDure} to $G$, we get (for some point
$x$ in the support of $v_{j,m}$, and some matrix $A$ of the
form $DF(R_n^{-1}(z))$ for some $z$)
  \begin{multline}\label{gla}
  \norm{\tilde v_{j,k,m} \circ R_n \circ F \circ R_n^{-1}}{H_p^{t,t_-}}
  \leq C \norm{u}{\HH_p^{0,t_-}}
  \\+\Cs \norm{\frac{\det A}{\det D G}}{C^\alpha} |\det A|^{-1/p}
  \max(\lambda_{u,n}(x)^{-t}, \lambda_{s,n}(x)^{-(t+t_-)})\norm{\tilde v_{j,k,m}}{H_p^{t,t_-}}.
  \end{multline}
The factor $\det A$ is close to $\det DT_\ii(x)^{-1}$.
Moreover, $\det DG=(\det DF)\circ R_n^{-1}$. By choosing
$r_n$ large enough, we can make sure that the $C^\alpha$ norm
of $\det DG$ is controlled by its sup norm, to ensure
that $\norm{\det A/\det DG}{C^\alpha}$ is uniformly
bounded.

Let $\chi'$ be a $C^\infty$ function supported in $O'$ with
$\chi'\equiv 1$ on the support of $v_{j,m}\circ T_\ii^{-1}$.
For $\delta>0$, we can ensure by increasing $r_n$ that  the
$C^\alpha$ norm of $(\chi' g^{(n)} )\circ T_\ii^{-1}
\circ\kappa_k^{-1}\circ R_n^{-1}$ is bounded by
$|g^{(n)}(x)|+\delta$ for some $x$ in the support of $v_{j,m}$.
Choosing $\delta>0$ small enough, we deduce from \eqref{gla},
Lemma~  \ref{Leib} and \eqref{v''}
  \begin{multline*}
  \norm{ (\rho_k(g^{(n)}1_{O_\ii} v_{j,m})\circ T_\ii^{-1})\circ \kappa_k^{-1}\circ R_n^{-1}}{H_p^{t,t_-}}
  \leq C \norm{u}{\HH_p^{0,t_-}}
  \\ + \Cs n \norm{ |\det DT^n|^{1/p} g^{(n)}\max(\lambda_{u,n}^{-t}, \lambda_{s,n}^{-(t+t_-)})}{L_\infty}\norm{\tilde v_{j,m}}{H_p^{t,t_-}}.
  \end{multline*}
This proves \eqref{eq:secondstep}.

\emph{Third step.} We have $\Lp_g^n u =\sum_{j,m} \sum_{\ii}
(1_{O_\ii} g^{(n)} v_{j,m})\circ T_\ii^{-1}$. (Note that only
finitely many terms in this sum are nonzero by compactness of
the support of each $\rho_j$.) We claim that the intersection
multiplicity of the supports of the functions $(1_{O_\ii}
g^{(n)} v_{j,m})\circ T_\ii^{-1}$ is bounded by $\Cs D_n^e$.
Indeed, this follows from the fact that any point $x\in X_0$
belongs to at most $D_n^e$ sets $\overline{T_\ii(O_\ii)}$, and
that the intersection multiplicity of the supports of the
functions $v_{j,m}$ is bounded.

To estimate $\norm{\Lp_g^n u}{n}$, we have to bound each
term $\norm{ (\rho_k \Lp_g^n u)\circ \kappa_k^{-1}\circ
R_n^{-1}}{H_p^{t,t_-}}$, for $1\leq k\leq J$. Let us fix such a
$k$. By Lemma \ref{lem:sum}, we have
  \begin{multline*}
  \norm{ (\rho_k \Lp_g^n u)\circ \kappa_k^{-1}\circ R_n^{-1}}{H_p^{t,t_-}}^p
  \leq C \norm{u}{\HH_p^{0,t_-}}^p
  \\ +\Cs (\Cs D_n^e)^{p-1}\sum_{j,m,\ii}
  \norm{(\rho_k(1_{O_\ii} g^{(n)} v_{j,m})\circ T_\ii^{-1})\circ \kappa_k^{-1}\circ
  R_n^{-1}}{H_p^{t,t_-}}^p.
  \end{multline*}
We can bound each term in the sum using \eqref{eq:secondstep}
and the convexity inequality $(a+b)^p \leq 2^{p-1}(a^p+b^p)$.
Moreover, for any $(j,m)$, the number of parameters $\ii$ for
which the corresponding term is nonzero is bounded by the
number of sets $\overline{O_\ii}$ intersecting the support of
$v_{j,m}$. Choosing $r_n$ large enough, we can ensure that the
supports of the $v_{j,m}$ are small enough so that this number is
bounded by $D_n^b$. Together with \eqref{eq:decompose}, this
concludes the proof of \eqref{eq:main}, and of Theorem
\ref{thm:MainSpectralThm}.
\end{proof}

%%%%%%%%%%%%%%%%%%%%%%%%%%%%%%%%
\appendix

\section{Corrigendum to \cite[Lemma 2.8]{baladi:Cinfty} -- About interpolation}
\label{whynotbaladi}

The statement of \cite[Lemma 2.8]{baladi:Cinfty} should be
replaced by\footnote{This has no consequences on the other
claims in  \cite{baladi:Cinfty}.}: letting $n=[|t|]
+[|t+t_-|]+d+4$, if $g$ is $C^{n}$, then
  \begin{equation}
  \norm{g u}{\HH_p^{t,t_-}} \leq
  \Cs\norm{g}{C^{n-1}(C^1_s)}
  \norm{u}{\HH_p^{t,t_-}} + C \norm{u}{\HH_p^{t-1,t_-}},
  \end{equation}
where $\norm{g}{C^{n-1}(C^1_s)}$ is the maximum between
$\norm{g}{L_\infty}$ and the  $C^{n-1}$ norm of the first
derivatives of $g$ along $E^s$. It was mistakenly claimed in
\cite[Lemma 2.8]{baladi:Cinfty} that it is enough to take
$n=3$. The sentence ``This can be shown by a
straightforward...oscillatory integral argument" in the proof
there should be replaced by ``This can be shown by integrating
by parts $[|p|]+[|q]|+d+1$ times in total with respect to
$(u,v)$, noting that
\begin{align*}
(1+|\eta-s\theta|^2+|\xi-s\omega|^2)^{p/2}&
(1+|\xi-s\omega|^2)^{q/2}(1+|\eta|^2+|\xi|^2)^{-p/2}(1+|\xi|^2)^{-q/2}
\\&\le
16  (1+|s\omega|^2)^{|q|/2} (1+ |s \theta|^2+ |s\omega|^2)^{|p|/2} \, .
\end{align*}
Since $\partial ^{\gamma''+\gamma'} h$ has been differentiated
up to $3$ times including $|\gamma'|\in \{ 1,2\}$
times along  $x$-directions, we get at most $[|p|]+[|q]|+d+4$ derivatives in total."
In particular \cite[Lemma 2.8]{baladi:Cinfty} only holds if $g$ is sufficiently
differentiable.

We derive  via interpolation in Lemma~ \ref{Leib} a simpler
Leibniz-type bound which takes the place of  \cite[Lemma
2.8]{baladi:Cinfty} and is valid for $g\in C^\alpha$ for any
$\alpha >0$. The ``zooming" norm (\ref{zoom}) then allows us to
replace $\norm{g}{C^\alpha}$ by a sup-norm type estimate for
arbitrary $g$.

The interpolation estimates also yield a chain-rule-type bound
(Lemma~ \ref{lem:CompositionDure} and  Remark ~ \ref{extend}) which extends
\cite[Lemma 2.10]{baladi:Cinfty} to arbitrary
differentiability: the proof of \cite[Lemma
2.10]{baladi:Cinfty} uses that $T$ is $C^\infty$ implicitly in
several places (when referring to arguments of \cite{AG}), although
a modification of this proof along the lines given above gives
the claim for $C^k$ dynamics, with $k(d)$ large if $d$ is
large.

\section{Properties of physical measures}
\label{sec:SRB}

\renewcommand{\LSRB}{\mathcal{L}}

In this section, we prove Theorem \ref{thm:ExistSRB}. In fact,
we will prove a more general result in a more abstract context.
Let $X$ be a manifold, $X_0$ a compact subset of $X$ with
positive Lebesgue measure, and $T:X_0 \to X_0$ a transformation
for which Lebesgue measure is nonsingular. We will denote in
this appendix by $\LSRB$ the corresponding transfer operator,
defined by duality on $L_1(\Leb)$ by $\int_{X_0} \LSRB f\cdot
g\dLeb=\int_{X_0} f\cdot g\circ T \dLeb$.
\begin{thm}
\label{thm:SRBabstrait} Let $H$ be a Banach space of
distributions supported on $X_0$. Assume that
\begin{enumerate}
\item There exist $\alpha>0$ and $C>0$ such that, for any $u\in H$
and $f\in C^\alpha(X)$, then $fu\in H$ and $\norm{fu}{H}\leq
C\norm{f}{C^\alpha}\norm{u}{H}$.
\item The space $H\cap L_\infty(\Leb)$ is dense in $H$.
\item The transfer operator $\LSRB$ associated to $T$ sends continuously
$H\cap L_\infty(\Leb)$ into itself, hence it admits a
continuous extension to $H$ (still denoted by $\LSRB$). We
assume that the essential spectral radius of this extension
is $<1$.
\item There exist $f_0\in H\cap L_\infty(\Leb)$ taking its values in $[0,1]$ and $N_0>0$ such that, for
any $\phi\in L_\infty(\Leb)$, then $f_0=1$ on the support
of $\LSRB^{N_0}\phi$.
\item For any $u\in H$ which is a limit of nonnegative functions $u_n\in H\cap L_\infty(\Leb)$
and for which there exists a measure $\mu_u$ such
that\footnote{We write $\langle u,g\dLeb\rangle$ and not
$\langle u, g\rangle$, in accordance with the convention
stated in the footnote page \pageref{page:foot}, viewing
distributions as generalized functions which can only be
integrated against smooth densities.} $\langle u, g\dLeb
\rangle=\int g\dd\mu_u$ for any $C^\infty$ function $g$,
then the measure $\mu_u$ gives zero mass to the
discontinuity set of $T$.
\end{enumerate}
Then there exist a finite number of probability measures
$\mu_1,\dots,\mu_l$ which are $T$-invariant and ergodic, and
disjoint sets $A_1,\dots,A_l$ such that $\mu_i(A_i)=1$,
$\Leb(A_i)>0$, $Leb(X_0\backslash \bigcup_{i=1}^l A_i)=0$ and,
for every $x\in A_i$ and every function $f\in
\overline{C^0(X_0)\cap H}$ (the closure of $C^0(X_0)\cap H$ in
$C^0(X_0)$), then $\frac{1}{n}\sum_{j=0}^{n-1}f(T^j x) \to \int
f\dd\mu_i$.

Moreover, for every $i$, there exist an integer $k_i$ and a
decomposition $\mu_i=\mu_{i,1}+\dots+\mu_{i,k_i}$ such that $T$
sends $\mu_{i,j}$ to $\mu_{i,j+1}$ for $j\in \Z/k_i\Z$, and the
probability measures $k_i\mu_{i,j}$ are exponentially mixing
for $T^{k_i}$ and $C^\alpha$ test functions.
\end{thm}
The proof will also describe a direct relationship between the
eigenfunctions of $\LSRB$ for eigenvalues of modulus $1$, and
the physical measures of $T$. The first part of the proof is
directly borrowed from \cite{BKL}.

The first, second and fourth conditions say that the space $H$
is sufficiently large. They are satisfied in the setting of
this paper (taking $f_0=1_{X_0}$, which belongs to
$\HH_p^{t,t_-}$), but also in the case of an attractor, when
$T(X_0)$ is contained in the interior of $X_0$ (the function
$f_0$ can be taken $C^\infty$, compactly supported in the
interior of $X_0$, equal to $1$ on $T(X_0)$).

The fifth condition is necessary,
as shown by Example \ref{ex2} in Section \ref{sec:examples}:
taking for $H$ the space of distributions in the Sobolev space
$H_2^{-1}$ supported in $[-1,1]\times \{0,1\}$, then all the
assumptions of the theorem but the fifth one are satisfied, and
the conclusion of the theorem does not hold.

\begin{proof}
Let us first prove the existence of $C>0$ such that, for any
$n\in \N$,
  \begin{equation}
  \label{eq:ItBounded}
  \norm{\LSRB^n}{H\to H}\leq C.
  \end{equation}
Otherwise, $\LSRB$ has an eigenvalue of modulus $>1$, or a
nontrivial Jordan block for an eigenvalue of modulus $1$. Let
$\lambda$ be an eigenvalue of $\LSRB$ of maximal modulus, with
a Jordan block of maximal size $d$. Since $L_\infty\cap H$ is
dense in $H$, its image under the eigenprojections is dense in
the eigenspaces, which are finite dimensional. Hence, it
coincides with the full eigenspaces. Therefore, there exists a
bounded function $f$ such that
$n^{-d}\sum_{i=0}^{n-1}\lambda^{-i} \LSRB^i f$ converges to a
nonzero limit $u$. For any $C^\infty$ function $g$,
  \begin{equation*}
  \langle u, g\dLeb\rangle = \lim \frac{1}{n^d}\sum_{i=0}^{n-1}
  \lambda^{-i} \langle \LSRB^i f, g\dLeb\rangle
  =\lim \frac{1}{n^d}\sum_{i=0}^{n-1}
  \lambda^{-i} \int f\cdot g\circ T^i \dLeb.
  \end{equation*}
If $|\lambda|>1$ or $d\geq 2$, this quantity converges to $0$
when $n\to \infty$ since $\int f\cdot g\circ T^i \dLeb$ is
uniformly bounded. This contradicts the fact that $u$ is
nonzero, and proves \eqref{eq:ItBounded}.

For $|\lambda|=1$, let $E_\lambda$ denote the corresponding
eigenspace, and $\Pi_\lambda:H\to E_\lambda$ the corresponding
eigenprojection. It is given by
  \begin{equation}
  \Pi_\lambda f = \lim \frac{1}{n}\sum_{i=0}^{n-1}\lambda^{-i} \LSRB^i f,
  \end{equation}
where the convergence holds in $H$. Since $L_\infty(\Leb)\cap
H$ is dense in $H$, $E_\lambda=\Pi_\lambda (L_\infty(\Leb)\cap
H)$. For any $f\in L_\infty(\Leb)\cap H$ and $g\in C^\infty$,
  \begin{equation}
  \left| \langle \Pi_\lambda f, g\dLeb\rangle \right| \leq
  \lim \frac{1}{n}\sum_{i=0}^{n-1} \left| \int f\cdot g\circ T^i\dLeb\right|
  \leq C \norm{f}{L_\infty}\norm{g}{C^0}.
  \end{equation}
By Riesz representation theorem on the compact space $X_0$,
this implies that, for any $u\in E_\lambda$, there exists a
finite measure $\mu_u$ on $X_0$ such that $\langle u,
g\dLeb\rangle = \int g\dd\mu_u$. Moreover, for $i\geq N_0$ and
$g\geq 0$,
  \begin{align*}
  \left| \int f\cdot g\circ T^i\dLeb\right|&
  =\left| \int \LSRB^{N_0}f\cdot g\circ T^{i-N_0} \dLeb\right|
  =\left| \int \LSRB^{N_0}f\cdot f_0 \cdot g\circ T^{i-N_0} \dLeb \right|
  \\&
  \leq C \int f_0\cdot g\circ T^{i-N_0} \dLeb
  =C\int \LSRB^{i-N_0} f_0 \cdot g\dLeb.
  \end{align*}
Averaging and taking the limit, we obtain
  \begin{equation}
  \left| \int g \dd\mu_{\Pi_\lambda f} \right| \leq C \int g \dd\mu_{\Pi_1 f_0}.
  \end{equation}
This means that the measures $\mu_u$ are all absolutely
continuous with respect to the reference measure
$\mu:=\mu_{\Pi_1 f_0}$, with bounded density.

Let us show that the measure $\mu$ is invariant. This is
formally trivial from the computation
  \begin{equation*}
  \int g\dd\mu=\langle \Pi_1 f_0, g\dLeb \rangle
  =\langle \LSRB \Pi_1 f_0, g\dLeb \rangle
  =\langle \Pi_1 f_0, g\circ T \dLeb\rangle
  =\int g\circ T \dd\mu.
  \end{equation*}
However, this argument is not correct since $\langle \Pi_1 f_0,
g\circ T \dLeb\rangle$ is not well defined since $g$ is not
smooth. More importantly, even if we could define it, the
equality between $\langle \Pi_1 f_0, g\circ T \dLeb\rangle$ and
$\int g\circ T\dd\mu$ would not be trivial since the
relationship between $\Pi_1 f_0$ and $\de\mu$ is established
only for continuous functions.

The rigorous proof relies on the fifth assumption of the
theorem. By definition, if $g$ is $C^\infty$, then $\int
g\dd\mu= \lim \int g\dd
\left(\frac{1}{n}\sum_{i=0}^{n-1}T_*^i(f_0\Leb)\right)$. By
density, this equality extends to $C^0$ functions, hence $\mu$
is the weak limit of the sequence of measures
$\frac{1}{n}\sum_{i=0}^{n-1}T_*^i(f_0\Leb)$. In turn, for any
function $h$ whose discontinuity set has zero measure for
$\mu$,
  \begin{equation}
  \label{eq:weak_limit}
  \int h\dd\mu= \lim \int h\dd
  \left(\frac{1}{n}\sum_{i=0}^{n-1}T_*^i(f_0\Leb)\right).
  \end{equation}
If $g$ is a continuous function, then $g\circ T$ is continuous
except on the discontinuity set of $T$. The fifth assumption of
the theorem shows that this set has zero measure for $\mu$.
Hence, \eqref{eq:weak_limit} applies to $g\circ T$. It also
applies to $g$. Since the right hand side for $g$ and $g\circ
T$ coincide up to $O(1/n)$, this yields $\int g\circ T
\dd\mu=\int g\dd\mu$ and concludes the proof of the invariance
of $\mu$.

In the following, we shall encounter several instances of
similar equations that are formally trivial but need a rigorous
justification. Let us give a last justification of this type,
and leave the remaining ones to the reader. We claim that, if
$\phi\in C^\alpha$ and $g\in C^\infty$,
  \begin{equation}
  \label{eq:lqskjdfmlk}
  \langle\LSRB^i( \phi \Pi_1 f_0), g\dLeb \rangle
  = \int \phi\cdot g\circ T^i \dd\mu.
  \end{equation}
Indeed, $\LSRB^i(\phi \Pi_1 f_0)$ is the limit in $H$ of
$\LSRB^i(\phi \frac{1}{n}\sum_{j=0}^{n-1} \LSRB^j f_0)$, hence
  \begin{align*}
  \langle\LSRB^i( \phi \Pi_1 f_0), g\dLeb \rangle&
  =\lim \frac{1}{n}\sum_{j=0}^{n-1} \langle \LSRB^i(\phi\LSRB^j f_0), g\dLeb \rangle
  \\&=\lim \frac{1}{n}\sum_{j=0}^{n-1} \int \phi \LSRB^j f_0 \cdot g\circ T^i \dLeb
  \\&
  =\lim  \int \phi \cdot g\circ T^i \dd\left(\frac{1}{n}\sum_{j=0}^{n-1}T_*^j(f_0\Leb)\right).
  \end{align*}
The measure $\mu$ gives zero mass to the discontinuities of
$g\circ T^i$ (since it is invariant and gives zero mass to the
discontinuities of $T$). Hence, \eqref{eq:weak_limit} holds for
$\phi \cdot g\circ T^i$. This concludes the proof of
\eqref{eq:lqskjdfmlk}.

\smallskip

For any $u\in E_\lambda$, write $\mu_u=\phi_u \mu$ where
$\phi_u \in L_\infty(\mu)$ is defined $\mu$-almost everywhere.
The equation $\LSRB u=\lambda u$ translates into $T_*(\phi_u
\mu)=\lambda \phi_u \mu$. Hence, since $\mu$ is invariant,
  \begin{align*}
  \int | \phi_u\circ T -\lambda^{-1} \phi_u|^2 \dd\mu&=
  \int |\phi_u|^2 \circ T \dd\mu + \int |\phi_u|^2 -2\Re
  \int \overline{\phi_u} \circ T  \lambda^{-1}\phi_u \dd\mu
  \\&
  =2 \int |\phi_u|^2 \dd\mu -2 \Re\int \lambda^{-1}\overline{\phi_u}\dd T_*(\phi_u \mu)
  =0.
  \end{align*}
Let $F_\lambda=\{ \phi \in L_\infty(\mu) \st \phi\circ
T=\lambda^{-1} \phi\}$ (this is a space of equivalence classes
of functions), then the map $\Phi_\lambda: u\mapsto \phi_u$
sends (injectively) $E_\lambda$ to $F_\lambda$. Let us show
that it is also surjective.

Let $\phi\in F_\lambda$. By Lusin's theorem, there exists a
sequence of $C^\alpha$ functions $\phi_p$ with
$\norm{\phi-\phi_p}{L_1(\mu)}\leq 1/p$. Let
$u_p=\Pi_\lambda(\phi_p \Pi_1 f_0)$, and let $\mu_p=\mu_{u_p}$.
Let us prove that the total mass of the measure $\phi\de\mu-
\de\mu_p$ converges to $0$. If $g$ is a $C^\infty$ function,
  \begin{align*}
  \int g\dd\mu_p&=\langle u_p, g\dLeb\rangle=\lim \frac{1}{n}\sum_{i=0}^{n-1}
  \lambda^{-i}\langle \LSRB^i (\phi_p \Pi_1 f_0), g\dLeb\rangle
  \\&  = \lim \frac{1}{n}\sum_{i=0}^{n-1}
  \lambda^{-i}\int  \phi_p \cdot g \circ T^i \dd\mu,
  \end{align*}
by \eqref{eq:lqskjdfmlk}. On the other hand, for any $n$, since
$\mu$ is invariant and $\phi\circ T=\lambda^{-1}\phi$,
  \begin{equation*}
  \int g \phi\dd\mu= \frac{1}{n}\sum_{i=0}^{n-1} \int g\circ T^i \phi\circ T^i \dd\mu
  = \frac{1}{n}\sum_{i=0}^{n-1} \lambda^{-i} \int g\circ T^i \phi \dd\mu.
  \end{equation*}
Subtracting the two previous equations, we get
  \begin{equation}
  \label{eq:oiqusfdioj}
  \left|\int g \phi\dd\mu -\int g\dd\mu_p \right| \leq \norm{
  \phi-\phi_p}{L_1(\mu)} \norm{g}{C^0},
  \end{equation}
which proves that the total mass of $\phi\de\mu- \de\mu_p$
converges to $0$.

The sequence $u_p$ belongs to the finite dimensional space
$E_\lambda$, and the elements of $E_\lambda$ are separated by
the linear forms given by the integration along $C^\infty$
densities (since $H$ is a space of distributions). Since
$\langle u_p,g\dLeb \rangle$ converges for any $g$, the
sequence $u_p$ is therefore converging to a limit $u_\infty$.
By construction, $\Phi_\lambda(u_\infty)=\phi$. This concludes
the proof of the surjectivity of $\Phi_\lambda$.

\smallskip

The eigenvalues of $\LSRB$ of modulus $1$ are exactly the
$\lambda$ such that $F_\lambda$ is not reduced to $0$. This set
is a group, since $\phi_\lambda \phi_{\lambda'} \in F_{\lambda
\lambda'}$ whenever $\phi_\lambda\in F_\lambda$ and
$\phi_{\lambda'}\in F_{\lambda'}$. Since $\LSRB$ only has a
finite number of eigenvalues of modulus $1$, this implies that
these eigenvalues are roots of unity. In particular, there
exists $N>0$ such that $\lambda^N=1$ for any eigenvalue
$\lambda$.

\smallskip

Let us now assume that $1$ is the only eigenvalue of $\LSRB$ of
modulus $1$ (in the general case, this will be true for
$\LSRB^N$, so we will be able to deduce the general case from
this particular case). Under this assumption, for any $u\in H$,
$\LSRB^n u$ converges to $\Pi_1 u$.

Consider the subset of $F_1$ given by the nonnegative functions
with integral $1$. It is a convex cone in $F_1$, whose extremal
points are of the form $1_B$ for some minimal invariant set
$B$. Such extremal points are automatically linearly
independent. Since $F_1$ is finite-dimensional, there is only a
finite number of them, say $1_{B_1},\dots, 1_{B_l}$, and a
function belongs to $F_1$ if and only if it can be written as
$\phi=\sum \alpha_i 1_{B_i}$ for some scalars
$\alpha_1,\dots,\alpha_l$. The decomposition of the function
$1\in F_1$ is given by $1=\sum 1_{B_i}$, hence the sets $B_i$
cover the whole space up to a set of zero measure for $\mu$.
Moreover, since $B_i$ is minimal, the measure
$\mu_i:=\frac{1_{B_i}\mu}{\mu(B_i)}$ is an invariant ergodic
probability measure.

Let $u_i=\Phi_1^{-1}(1_{B_i}) \in H$, then any element of $E_1$
is a linear combination of the $u_i$. In particular, this
applies to $\Pi_1(fu_i)$ for any $f\in C^\alpha$. Let us show that
  \begin{equation}
  \label{eq:Exprimeai}
  \Pi_1(fu_i)=\left(\int f\dd\mu_i\right) u_i.
  \end{equation}
We can write $\Pi_1(fu_i)=\sum a_{ij}(f)u_j$. Let us fix once
and for all $l$ sequences of $C^\alpha$ functions $\phi_{j,p}$
taking values in $[0,1]$ and such that $\phi_{j,p}$ converges
in $L_1(\mu)$ to $1_{B_j}$. Since $\langle u_j,
\phi_{j',p}\dLeb \rangle=\int_{B_j} \phi_{j',p}\dd\mu \to
\delta_{jj'}\mu(B_j)$, we have $a_{ij}(f)=\frac{1}{\mu(B_j)}
\lim_{p\to\infty} \langle \Pi_1 (fu_i), \phi_{j,p}\dLeb
\rangle$. Moreover, if $p$ is fixed,
  \begin{align*}
  \langle \Pi_1 (fu_i), \phi_{j,p}\dLeb \rangle&
  =\lim_{n\to \infty} \langle \LSRB^n (fu_i), \phi_{j,p}\dLeb \rangle
  \\&
  =\lim_{n\to \infty} \int_{B_i} f \phi_{j,p}\circ T^n \dd\mu.
  \end{align*}
Writing $\phi_{j,p}\circ T^n=1_{B_j}\circ T^n+
(\phi_{j,p}-1_{B_j})\circ T^n$ and using $1_{B_j}\circ
T^n=1_{B_j}$ and $\norm{(\phi_{j,p}-1_{B_j})\circ
T^n}{L_1(\mu)}=\norm{\phi_{j,p}-1_{B_j}}{L_1(\mu)}\to_{p\to\infty}
0$, we obtain \eqref{eq:Exprimeai}.

This enables us to deduce that each measure $\mu_i$ is
exponentially mixing, as follows. Let $\delta<1$ be such that
$\norm{\LSRB^n -\Pi_1}{H\to H}=O(\delta^n)$. Then, if $f,g$ are
$C^\alpha$ functions,
  \begin{align*}
  \int f\cdot g\circ T^n \dd\mu_i&
  =\frac{1}{\mu(B_i)}\langle \LSRB^n(fu_i), g \dLeb \rangle
  =\frac{1}{\mu(B_i)}\langle \Pi_1(fu_i), g\dLeb \rangle+O(\delta^n)
  \\&
  =\left(\int f\dd\mu_i\right)\frac{1}{\mu(B_i)} \langle u_i, g\dLeb\rangle+O(\delta^n)
  \\&
  =\left(\int f\dd\mu_i\right)\left(\int g\dd\mu_i\right)+O(\delta^n).
  \end{align*}

\smallskip

We now turn to the relationships between Lebesgue measure and
the measures $\mu_i$. For any function $f\in L_\infty(\Leb)\cap
H$, let us write
  \begin{equation}
  \Pi_1(f)=\sum_{i=1}^l b_i(f) u_i.
  \end{equation}
We will need to describe the coefficients $b_i(f)$. Let $n_p$
be a sequence tending fast enough to $\infty$ so that
$\norm{\Lp^{n_p} -\Pi_1}{H\to H}
\norm{\phi_{i,p}}{C^\alpha}\to_{p\to\infty} 0$. If $f$ belongs to
$L_\infty(\Leb) \cap H$,
  \begin{align*}
  \int f\cdot  \phi_{i,p}\circ T^{n_p} \dLeb &= \langle \LSRB^{n_p}f, \phi_{i,p}\dLeb \rangle
  \\&
  = \langle \phi_{i,p} (\LSRB^{n_p}-\Pi_1)f,\dLeb \rangle
  + \langle \Pi_1 f, \phi_{i,p} \dLeb \rangle
  \\&
  = o(1) + \sum_{j=1}^l b_j(f) \int_{B_j} \phi_{i,p} \dd\mu
  = o(1) + b_i(f)\mu(B_i).
  \end{align*}
More generally, $\int f\cdot
\left(\frac{1}{n_p}\sum_{n=n_p}^{2n_p-1}\phi_{i,p} \circ
T^n\right)\dLeb \to \mu(B_i)b_i(f)$. The sequence
$\frac{1}{n_p}\sum_{n=n_p}^{2n_p-1}\phi_{i,p} \circ T^n$ is
bounded in $L^2(\Leb)$, and asymptotically invariant. Let
$h_i:X\to [0,1]$ be one of its weak limits. It satisfies
  \begin{equation}
  b_i(f)= \frac{1}{\mu(B_i)} \int f h_i \dLeb,
  \end{equation}
and $h_i\circ T=h_i$. Since $b_i(f_0)=1$, we have $\int h_i f_0
\dLeb=\mu(B_i)$.

Let us now compute $\int h_ih_j f_0\dLeb$. We have
  \begin{align*}
  \mu(B_j)b_j(\phi_{i,p}\LSRB^{n} f_0) &= \int \phi_{i,p} \LSRB^{n} f_0 h_j\dLeb
  = \int f_0 \phi_{i,p}\circ T^{n} h_j\circ T^{n}\dLeb
  \\&
  =\int\phi_{i,p} \circ T^{n} h_j f_0\dLeb.
  \end{align*}
Taking the average and the weak-limit, we obtain
  \begin{equation}
  \int h_i h_j f_0\dLeb=\mu(B_j)\lim_{p\to \infty} \frac{1}{n_p}\sum_{n=n_p}^{2n_p-1}
  b_j( \phi_{i,p} \LSRB^n f_0).
  \end{equation}
Moreover, if $n\geq n_p$,
  \begin{equation}
  \phi_{i,p} \LSRB^{n}f_0=\phi_{i,p} (\LSRB^{n} -\Pi_1) f_0 + \phi_{i,p} \Pi_1 f_0.
  \end{equation}
The first term converges to $0$ in $H$, and the computation
made in \eqref{eq:oiqusfdioj} shows that $\Pi_1(\phi_{i,p}
\Pi_1 f_0)$ converges to $u_i$. This implies that $b_j(
\phi_{i,p} \LSRB^{n}f_0)$ converges to $\delta_{ij}$. This
yields
  \begin{equation}
  \label{eq:orthogonal}
  \int h_i h_j f_0\dLeb=\mu(B_j)\delta_{ij}.
  \end{equation}
Let $X_1=\{x\st f_0(x)>0\}$. Taking $i=j$, we get $\int h_i^2
f_0\dLeb=\mu(B_i)=\int h_i f_0\dLeb$. Since $h_i$ takes its
values in $[0,1]$, this shows that there exists a subset
$C_i^0$ of $X_1$ such that $h_i 1_{X_1}=1_{C_i^0}$, with
$\int_{C_i^0} f_0 \dLeb=\mu(B_i)$. Moreover,
\eqref{eq:orthogonal} shows that $\Leb(C^0_i\cap C^0_j)=0$ if
$i\not=j$. Let $C_i=T^{-N_0}C_i^0$, then these sets are
disjoint. For any function $f\in L_\infty(\Leb)\cap H$, since
$\LSRB^{N_0}f$ is supported in $X_1$,
  \begin{align*}
  b_i(f)&=b_i(\LSRB^{N_0} f)=\frac{1}{\mu(B_i)} \int \LSRB^{N_0}f h_i \dLeb
  =\frac{1}{\mu(B_i)} \int \LSRB^{N_0} f\cdot 1_{C_i^0} \dLeb
  \\&=\frac{1}{\mu(B_i)} \int f \cdot 1_{C_i^0}\circ T^{N_0} \dLeb
  =\frac{1}{\mu(B_i)} \int_{C_i} f\dLeb.
  \end{align*}
Moreover, since $\LSRB^{N_0}1$ is supported on the sets
$C_i^0$,
  \begin{align*}
  \Leb(X_0)&=\int 1 \dLeb=\int \LSRB^{N_0}1 \dLeb= \int \LSRB^{N_0}1 \cdot 1_{\bigcup C_i^0} \dLeb
  \\&= \int 1_{\bigcup C^0_i} \circ T^{N_0}\dLeb
  =\int 1_{\bigcup C_i} \dLeb.
  \end{align*}
This shows that the sets $C_i$ form a partition of the space
modulo a set of zero Lebesgue measure. We have proved that
  \begin{equation}
  \Pi_1(f)=\sum_{i=1}^l \frac{ \int_{C_i} f\dLeb}{\mu(B_i)} u_i.
  \end{equation}

Let us now turn to the convergence of $\frac{1}{n}
\sum_{j=0}^{n-1}f\circ T^j$, for $f\in L_\infty(\Leb)\cap H$.
Let $S_n f=\sum_{j=0}^{n-1} f\circ T^j$, we will estimate $\int
|S_n f/n -S_m f/m|^2 f_0\dLeb$. For $i,j\geq 0$, we have
  \begin{align*}
  \int f\circ T^i \cdot f&\circ T^{i+j} f_0\dLeb
  = \int f\LSRB^i(f_0)
  \cdot f\circ T^j \dLeb
  \\&
  =\int \LSRB^j( f\LSRB^i f_0) f\dLeb
  = \langle \LSRB^j( f\LSRB^i f_0), f\rangle
  \\&
  = \langle \LSRB^j ( f \Pi_1 f_0), f\rangle + O(\delta^i)
  = \langle \Pi_1 (f\Pi_1 f_0), f\rangle + O(\delta^i)+O(\delta^j),
  \end{align*}
where $\delta<1$ is given by the spectral gap of the operator
$\LSRB$. Hence, for $n,m>0$,
  \begin{align*}
  \int S_nf& \cdot S_mf f_0\dLeb
  \\&=nm\langle \Pi_1 (f\Pi_1 f_0), f\rangle+
  \sum_{\substack{0\leq i\leq n-1\\ 0\leq j \leq m-1-i}} O(\delta^i)+O(\delta^j)
  + \sum_{\substack{0\leq i\leq m-1\\ 0<j\leq n-1-i}} O(\delta^i)+O(\delta^j)
  \\&
  =nm\langle \Pi_1 (f\Pi_1 f_0), f\rangle + O(n)+O(m).
  \end{align*}
Expanding the square in $|S_n f/n-S_m f/m|^2$, we get using the
previous equation
  \begin{align*}
  \int &|S_n f/n -S_m f/m|^2 f_0\dLeb\\&=
  \frac{1}{n^2}\int S_nf\cdot S_n f f_0\dLeb+ \frac{1}{m^2}\int S_mf\cdot S_mf f_0\dLeb
  -\frac{2}{nm} \int S_nf \cdot S_mf f_0\dLeb
  \\&=
  O(1/n)+O(1/m).
  \end{align*}

The functions $g_p=S_{p^4}f/p^4$ therefore satisfy
$\norm{g_{p+1}-g_p}{L_2(f_0\dLeb)}=O(1/p^2)$, which is
summable. This implies that $g_p$ converges in $L_2(f_0\dLeb)$
and almost everywhere for this measure. For a general $n\in
\N$, let $p$ be such that $p^4\leq n <(p+1)^4$, then $S_n f/n -
S_{p^4}f/p^4$ is uniformly small if $n$ is large. Hence, $S_n
f/n$ converges almost everywhere and in $L_2(f_0\dLeb)$, to a
function $\phi_f\in L_2(f_0\dLeb)$.

Let us now identify the function $\phi_f$. For any smooth
function $\phi$,
  \begin{align*}
  \int \phi \cdot f\circ T^n &f_0\dLeb = \langle \LSRB^n (\phi f_0), f\dLeb \rangle
  \\& \to \langle \Pi_1 (\phi f_0), f\dLeb \rangle
  = \sum_{i=1}^l b_i(\phi f_0) \int_{B_i} f\dd\mu
  \\&\ \ \
  = \sum_{i=1}^l \frac{\int_{C_i} \phi f_0\dLeb}{\mu(B_i)}\int_{B_i} f\dd\mu
  = \int \left(\sum_{i=1}^l 1_{C_i} \frac{\int_{B_i} f\dd\mu}{\mu(B_i)}\right) \phi f_0\dLeb.
  \end{align*}
This shows that, with respect to the measure $f_0\dLeb$, the
sequence of functions $f\circ T^n$ converges weakly to the
function $\tilde\phi_f:=\sum_{i=1}^l 1_{C_i} \left(\int
f\dd\mu_i\right)$. In turn, $S_n f/n$ converges weakly to
$\tilde\phi_f$. However, $S_nf/n$ converges strongly to
$\phi_f$, hence $\phi_f=\tilde \phi_f$ almost everywhere for
$f_0\dLeb$, and in particular on almost all $\bigcup_{i=1}^l
C_i^0$.

Let $A_i^f$ be the set of points for which $S_n f/n$ converges
to $\int f\dd\mu_i$. We have shown that $A_i^f$ contains a full
Lebesgue measure subset of $C^0_i$. However, $A_i^f$ is
$T$-invariant, hence it contains a full Lebesgue measure subset
of $C_i$. Since the sets $C_i$ cover Lebesgue almost all the
space, $\Leb(X\backslash \bigcup_{i=1}^l A_i^f)=0$. By the Birkhoff
ergodic theorem, $A^f_i$ is also a full $\mu$ measure subset of
$B_i$. Let $f_n$ be a countable sequence of functions in
$C^0(X_0)\cap H$, which is $C^0$-dense in
$\overline{C^0(X_0)\cap H}$, and set $A_i=\bigcap_{n\in \N}
A^{f_n}_i$. These sets satisfy the conclusion of the theorem.

\smallskip

This concludes the proof of the theorem when $1$ is the only
eigenvalue of modulus $1$ of $\LSRB$. If $\LSRB$ has other
eigenvalues of modulus $1$, let $N$ be such that $\lambda^N=1$
for all these eigenvalues $\lambda$. The above result applies
to $T^N$, and gives sets $A_1,\dots,A_l$ and probability
measures $\mu_1,\dots,\mu_l$. The map $T$ induces a permutation
of the sets $A_i$ (modulo sets of $0$ measure for $\mu$), say
$T(A_i)=A_{\sigma(i)} \mod 0$ for some permutation $\sigma$ of
$\{1,\dots,l\}$. For any orbit $(i_1,\dots,i_k)$ of $\sigma$,
the measure $\frac{1}{k}(\mu_{i_1}+\dots+\mu_{i_k})$ is
$T$-invariant, and its basin of attraction contains
$\bigcap_{j=0}^{N-1} T^{-j}(A_{i_1}\cup\dots\cup A_{i_k})$.
These measures are the measures of the statement of the
theorem, and their properties readily follow from the
corresponding properties for $T^N$.
\end{proof}

To deduce Theorem \ref{thm:ExistSRB} from Theorem~\ref{thm:SRBabstrait},
we just have to check the fifth
condition of  Theorem~\ref{thm:SRBabstrait} since the other ones are trivially
satisfied. Working locally in a chart, it is sufficient to
prove the following lemma:
\begin{lem}\label{deduce}
Let $K$ be a compact smooth hypersurface with boundary in
$\R^d$, whose intersection with almost every line parallel to a
coordinate axis has at most $L<\infty$ points. Let
$1/p-1<t_-\leq 0\leq t<1/p$, and let $u\in H_p^{t,t_-}$ be such
that
\begin{itemize}
\item there exists a sequence of nonnegative functions
$u_n\in H_p^{t,t_-}\cap L_\infty(\Leb)$ converging in
$H_p^{t,t_-}$ to $u$.
\item there exists a measure $\mu$ with $\langle u, g\dLeb\rangle = \int
g\dd\mu$ for any $C^\infty$ function $g$.
\item The support of $u$ does not intersect $\partial K$.
\end{itemize}
Then $\mu(K)=0$.
\end{lem}
\begin{proof}
Let us first prove that there exists a sequence of
neighborhoods $K_n$ of $K\cap \supp u$, whose intersection with
almost every line parallel to a coordinate axis has at most
$L'<\infty$ connected components, and with $\Leb(K_n)\to 0$.

Working locally, we can assume that $K$ is transversal to a
coordinate direction, say the last one. Hence, we can assume
that $u$ is supported in $[-1/2,1/2]^{d-1}\times \R$, and that
$K$ can be written as the graph of a smooth function $f$,
  \begin{equation}
  K=\{(x_1,\dots,x_{d-1}, f(x_1,\dots,x_{d-1})) \st
  (x_1,\dots,x_{d-1})\in [-1,1]^{d-1}\}.
  \end{equation}
Let $K_n=\{(x_1,\dots,x_{d-1}, f(x_1,\dots,x_{d-1})+y) \st
(x_1,\dots,x_{d-1})\in [-1,1]^{d-1}, |y|<1/n\}$. It is a
neighborhood of $K\cap \supp u$. It intersects any line
parallel to the last coordinate axis along one connected
component. Consider now another coordinate axis, say the first
one. Fix $(x_2,\dots,x_{d-1})$. Then the boundary of $K_n \cap
(\R\times \{(x_2,\dots,x_{d-1})\} \times \R)$ is formed of two
vertical segments and two translates of the graph of the
function $x \mapsto f(x,x_2,\dots,x_d)$. For almost every
$(x_2,\dots,x_d)$, this graph intersects almost every
horizontal line along at most $L$ points. Hence, the
intersection of almost every horizontal line with the boundary
of $K_n \cap (\R\times \{(x_2,\dots,x_{d-1})\} \times \R)$ has
at most $2L+2$ points. In particular, $K_n$ intersects almost
every horizontal line along at most $2L+1$ connected
components. This concludes the construction of $K_n$.

\smallskip

By Lemma \ref{lem:multiplier}, there exists a constant $C$ such
that, for any $n\in \N$, the multiplication by $1_{K_n}$ sends
$H_p^{t,t_-}$ into itself, with a norm bounded by $C$. In
particular, $1_{K_n}$ belongs to $H_p^{t,t_-}$ and is bounded
in this space.

Let us show that $1_{K_n}$ tends to $0$ in $H_p^{t,t_-}$. Let
$t'\in (t, 1/p)$. Then $1_{K_n}$ is also bounded in
$H_p^{t',t_-}$ by the same argument. Since the injection of
$H_p^{t',t_-}$ in $H_p^{t,t_-}$ is compact, the sequence
$1_{K_n}$ is therefore relatively compact in $H_p^{t,t_-}$. Let
$v$ be one of its cluster values. For any smooth function $g$,
  \begin{equation}
  \langle v, g\dLeb\rangle=\lim \langle 1_{K_n}, g\dLeb \rangle
  =\lim \int  1_{K_n} g\dLeb =0,
  \end{equation}
since $\Leb(K_n)$ tends to $0$. Hence, $v$ is the zero
distribution. The sequence $1_{K_n}$ is relatively compact in
$H_p^{t,t_-}$ and its only cluster value is zero, hence it
converges to $0$.

Let us now show that, for any $v\in H_p^{t,t_-}$,
  \begin{equation}
  \label{eq:opqousfoj}
  \norm{1_{K_n} v}{H_p^{t,t_-}}\to 0.
  \end{equation}
Choose a $C^\infty$ function $\phi$ with
$\norm{v-\phi}{H_p^{t,t_-}}\leq \epsilon$, then
  \begin{align*}
  \norm{1_{K_n} v}{H_p^{t,t_-}}&\leq \norm{1_{K_n}(v-\phi)}{H_p^{t,t_-}} + \norm{1_{K_n} \phi}{H_p^{t,t_-}}
  \\&\leq C \norm{v-\phi}{H_p^{t,t_-}} + \norm{\phi}{C^1} \norm{1_{K_n}}{H_p^{t,t_-}}
  \leq C\epsilon+o(1).
  \end{align*}
This proves \eqref{eq:opqousfoj}.

Let $g$ be a $C^\infty$ function supported in $K_n$, taking its
values in $[0,1]$, equal to $1$ on $K$. We claim that
  \begin{equation}
  \label{eq:oiuwoijxwcl}
  \int g\dd\mu \leq \langle 1_{K_n} u, \dLeb\rangle.
  \end{equation}
Indeed, write $u=\lim u_m$ where $u_m$ is a nonnegative
function belonging to $L_\infty(\Leb) \cap H_p^{t,t_-}$. Then
$\langle u_m, g\dLeb\rangle=\int gu_m\dLeb\leq \int 1_{K_n}
u_m\dLeb=\langle 1_{K_n} u_m, \dLeb \rangle$. Taking the limit
over $m$, we get \eqref{eq:oiuwoijxwcl}.

We can now conclude the proof: by \eqref{eq:oiuwoijxwcl}, we
have $\mu(K) \leq C\norm{1_{K_n} u}{H_p^{t,t_-}}$. This
quantity converges to $0$ by \eqref{eq:opqousfoj}.
\end{proof}

\begin{rmk}\label{nodirac}
The proof of the previous lemma implies that Dirac masses
cannot belong to $H^{t,t_-}_p$ if $1/p-1<t_-\leq 0\leq t<1/p$:
assume for a contradiction that $\delta_0$, the Dirac mass at
$0$, belongs to  $H^{t,t_-}_p$. Take $K_n$ the ball of radius
$1/n$ centered at $0$. Then $\delta_0 =1_{K_n} \delta_0$ for
each $n$, but $1_{K_n}\delta_0$ tends to zero in $H^{t,t_-}_p$
as $n\to \infty$, a contradiction.
\end{rmk}

%\bibliography{biblio}

\begin{thebibliography}{10}

\bibitem{AG}
S. Alinhac and P. G\'{e}rard, \textit{Op\'{e}rateurs pseudo-diff\'{e}rentiels et th\'{e}or\`{e}me de
Nash-Moser,} CNRS \'{e}ditions, Paris, 1991.

\bibitem{Bagh} A.G. Baghdasaryan,
\textit{On the intermediate spaces and quasi-linearizability of a pair
of spaces of Sobolev-Liouville type,} (russian)
Analysis Math. {\bf 24} (1998) 3--14.

\bibitem{BB}
M. Baillif, V. Baladi, \textit{Kneading determinants and spectra of transfer operators in higher dimensions: the isotropic case,}
Ergodic Theory Dynam. Systems  {\bf 25}  (2005) 1437--1470.


\bibitem{baladi:Cinfty} V. Baladi,
\textit{Anisotropic Sobolev spaces and dynamical transfer operators: $C^\infty$ foliations,}
in:  Algebraic and Topological Dynamics, S. Kolyada, Y. Manin \& T. Ward (eds). Contemporary Mathematics, Amer. Math. Society, (2005) 123--136.

\bibitem{BT1} V. Baladi and M. Tsujii,
\textit{Anisotropic H\"{o}lder and Sobolev spaces for hyperbolic
diffeomorphisms,} Annales de l'Institut Fourier {\bf 57} (2007)
127--154.

\bibitem{BT2} V. Baladi and M.Tsujii,
\textit{Dynamical determinants and spectrum for hyperbolic diffeomorphisms,}
{\tt arxiv.org} preprint (2006).


\bibitem{BKL} M. Blank, G. Keller, and C. Liverani,
\textit{Ruelle-Perron-Frobenius spectrum for Anosov maps,}
Nonlinearity {\bf 15} (2002), 1905--1973.
%\bibitem{BKS} Brudnyi, Krein et Semenov (MR0887950)???

\bibitem{Bu} J. Buzzi,
\textit{Intrinsic ergodicity of affine maps in $[0,1]\sp d$,}
Monatsh. Math.  {\bf 124}  (1997) 97--118.

\bibitem{Bu00} J. Buzzi,
\textit{Absolutely continuous invariant probability measures for arbitrary expanding piecewise
$\mathbf R$-analytic mappings of the plane,}
Ergodic Theory Dynam. Systems {\bf 20} (2000) 697--708.

\bibitem{Bu01} J. Buzzi,
\textit{No or infinitely many a.c.i.p. for piecewise expanding $C\sp r$ maps in higher dimensions,}
Comm. Math. Phys. {\bf 222} (2001) 495--501

\bibitem{BuMa}
J. Buzzi and V.  Maume-Deschamps,
\textit{Decay of correlations for piecewise invertible maps in higher dimensions,}  Israel J. Math.  {\bf 131}  (2002) 203---220.

\bibitem{Ch0} N. I. Chernov, \textit{Ergodic and statistical properties of piecewise linear hyperbolic automorphisms of the 2-torus,} J. of Stat. Phys. {\bf 69} (1992) 111--134.


\bibitem{Ch}
N. Chernov \textit{Decay of correlations and dispersing
billiards,} J. Statist. Phys. {\bf 94} (1999) 513--556.


\bibitem{CE} P. Collet and J.-P. Eckmann,
\textit{Liapunov multipliers and decay of correlations in dynamical systems,}
J. Statist. Phys. {\bf 115} (2004) 217--254.

\bibitem{Co} W.J. Cowieson,
\textit{Absolutely continuous invariant measures for most piecewise smooth expanding maps,}
 Ergodic Theory Dynam. Systems  {\bf 22}  (2002) 1061--1078.

\bibitem{DL} M. Demers and C. Liverani,
\textit{Stability of statistical properties in two-dimensional
piecewise hyperbolic maps,}
{\tt arxiv.org} preprint (2006) to appear in Transactions Amer. Math. Soc.



\bibitem{GB}
P. G\'{o}ra and A. Boyarsky, \textit{Absolutely continuous
invariant measures for piecewise expanding $C\sp 2$
transformation in $R\sp N$,} Israel J. Math. {\bf 67} (1989)
272--286.

\bibitem{GL1} S. Gou\"{e}zel and C. Liverani,
\textit{Banach spaces adapted to Anosov systems,}
Ergodic Theory Dynam. Systems {\bf 26} (2006) 189--218.

\bibitem{GL2} S. Gou\"{e}zel and C. Liverani,
\textit{Compact locally maximal hyperbolic sets for smooth maps: fine statistical properties,}
{\tt arxiv.org} preprint (2006).

\bibitem{He} H. Hennion,
\textit{ Sur un th\'{e}or\`{e}me spectral
et son application aux noyaux lipschitziens,}
Proc. Amer. Math. Soc.  118 (1993) 627--634.

\bibitem{Ke} G. Keller,
\textit{Ergodicit\'{e} et mesures invariantes pour les transformations dilatantes par morceaux d'une r\'{e}gion born\'{e}e du plan,}
C. R. Acad. Sci. Paris S\~{A}r. A-B {\bf 289} (1979) A625--A627.


\bibitem{RS} T. Runst and W. Sickel,
\textit{Sobolev spaces of fractional order, Nemytskij
operators, and nonlinear partial differential equations,}
Walter de Gruyter \& Co., Berlin (1996).



\bibitem{Sau} B. Saussol,
\textit{ACIMs for multidimensional expanding maps,}
Israel J. Math. {\bf 116} (2000) 223--248.



\bibitem{Str}
R. Strichartz, \textit{Multipliers on fractional Sobolev
spaces,} J. Math. Mech. {\bf 16} (1967) 1031--1060.



\bibitem{Tr} H. Triebel,
\textit{General function spaces III (spaces $B^{g(x)}_{p,q}$
and $F^{g(x)}_{p,q}$, $1< p < \infty$: basic properties),}
Analysis Math. {\bf 3} (1977) 221--249.

\bibitem{TrB} H. Triebel,
\textit{Interpolation Theory, Function Spaces, Differential
Operators,} North Holland, Amsterdam (1978).

\bibitem{Trie} H. Triebel, \textit{Theory of function spaces II,}
Birkh\"auser, Basel
(1992)

\bibitem{Ts00} M. Tsujii,
\textit{Piecewise expanding maps on the plane with singular ergodic properties,}
Ergodic Theory Dynam. Systems {\bf 20} (2000) 1851--1857.

\bibitem{Ts00'} M. Tsujii,
\textit{Absolutely continuous invariant measures for piecewise real-analytic expanding maps on the plane,}
Comm. Math. Phys. {\bf 208} (2000) 605--622.

\bibitem{Tsujii} M. Tsujii, \textit{Decay of correlations in
suspension semi-flows of angle-multiplying maps}, {\tt
arxiv.org} preprint (2006).

\bibitem{VoPa} L.R. Volevich and B.P. Paneyakh,
\textit{Certain spaces of generalised functions and embedding
theorems,} Uspekhi Mat. Nauk. {\bf 20} (1965) 3--74 (english
translation Russian Math. Surveys {\bf 20} 1965).


\bibitem{Yo} L.-S. Young,
\textit{Statistical properties of dynamical systems with some
hyperbolicity,} Ann. of Math. {\bf 147} (1998) 585--650.

\end{thebibliography}
\bibliographystyle{alpha}

\end{document}